\documentclass{article}
\usepackage{amsthm}
\usepackage{amsmath}
\usepackage{amssymb}
\usepackage{graphicx}
\usepackage{float}
\usepackage[font=small,labelsep=none]{caption}
\usepackage{subcaption}
\usepackage{pbox}
\usepackage{adjustbox}
\usepackage{url}

\newtheorem{theorem}{Theorem}[section]
\newtheorem{lemma}[theorem]{Lemma}

\newtheorem{corollary}[theorem]{Corollary}

\makeatletter
\let\@fnsymbol\@arabic
\makeatother

\def\jnl#1{\textit{\frenchspacing #1}}

\def\edel{\backslash}
\def\vdel{-}
\def\sdel{-}
\def\wheelfive{W_5}
\def\ctr{/}
\let\implies\Rightarrow
\let\implied\Leftarrow
\let\isom\cong
\def\G{\mathcal{G}}
\def\Gi{{\widetilde{\mathcal{G}}}}

\def\hangitem{\par\hangindent=2\parindent\hangafter=1}
\def\ptwo{R}

\begin{document}

\title{\bf A characterization of $K_{2,4}$-minor-free graphs}

\author{%
  M. N. Ellingham%
	\thanks{Supported by National Security Agency grant
H98230-13-1-0233 and Simons Foundation award 245715.  The United States
Government is authorized to reproduce and distribute reprints
notwithstanding any copyright notation herein.}%
	\\
  Department of Mathematics, 1326 Stevenson Center\\
  Vanderbilt University, Nashville, Tennessee 37212, U.S.A.\\
  \texttt{mark.ellingham@vanderbilt.edu}%
 \and
  Emily A. Marshall${}^{1,}$%
	\footnote{Work in this paper was done while at Vanderbilt
University.}%
	\\
  Department of Mathematics, 303 Lockett Hall\\
  Louisiana State University, Baton Rouge, Louisiana 70803, U.S.A.\\
  \texttt{emarshall@lsu.edu}
 \and
  \hbox to 0.8\hsize{\hss
    Kenta Ozeki%
	\footnote{Supported in part by JSPS KAKENHI Grant Number 25871053 %
	and by a Grant for Basic Science Research Projects from The %
	Sumitomo Foundation.}%
	\hss}\\
National Institute of Informatics,\\
2-1-2 Hitotsubashi, Chiyoda-ku, Tokyo 101-8430, Japan \\
and \\
  JST, ERATO, Kawarabayashi Large Graph Project, Japan\\
  \texttt{ozeki@nii.ac.jp}%
 \and
  Shoichi Tsuchiya%
	\\
  School of Network and Information, Senshu University,\\
	2-1-1 Higashimita, Tama-ku, Kawasaki-shi, Kanagawa, 214-8580, Japan\\ 
	\texttt{s.tsuchiya@isc.senshu-u.ac.jp}%
}

\maketitle

 \begin{abstract}
 We provide a complete structural characterization of
$K_{2,4}$-minor-free graphs.
 The $3$-connected $K_{2,4}$-minor-free graphs consist of nine small
graphs on at most eight vertices, together with a family of planar
graphs that contains $2n-8$ nonisomorphic
graphs of order $n$ for each $n \geq 5$ as well as $K_4$.
 To describe the $2$-connected $K_{2,4}$-minor-free graphs we
use \textit{$xy$-outerplanar} graphs, graphs embeddable in the plane
with a Hamilton $xy$-path so that all other edges lie on one side of
this path.
 We show that, subject to an appropriate connectivity condition,
$xy$-outerplanar graphs are precisely the graphs that have no rooted
$K_{2,2}$ minor where $x$ and $y$ correspond to the two vertices on one
side of the bipartition of $K_{2,2}$.
 Each $2$-connected $K_{2,4}$-minor-free graph is then (i) outerplanar,
(ii) the union of three $xy$-outerplanar graphs and possibly the edge
$xy$, or (iii) obtained from a $3$-connected $K_{2,4}$-minor-free graph
by replacing each edge $x_iy_i$ in a set $\{x_1 y_1, x_2 y_2, \ldots,
x_k y_k\}$ satisfying a certain condition by an $x_i y_i$-outerplanar
graph.
 From our characterization it follows that a
$K_{2,4}$-minor-free graph has a hamilton cycle if it is $3$-connected
and a hamilton path if it is $2$-connected.  Also, every $2$-connected
$K_{2,4}$-minor-free graph is either planar, or else toroidal and
projective-planar.
 \end{abstract}

\section{Introduction}

 The Robertson-Seymour Graph Minors project has shown that minor-closed
classes of graphs can be described by finitely many forbidden
minors.  Excluding a small number of minors can give
graph classes with interesting properties.
 The first such result was Wagner's demonstration \cite{wagner} that
planar graphs are precisely the graphs that are $K_5$- and
$K_{3,3}$-minor-free.

 Excluding certain special classes of graphs as minors seems to give
close connections to other graph properties.  One of the most important
open problems at present is Hadwiger's Conjecture, which relates
excluded complete graph minors to chromatic number.
 Our interest is in excluding complete bipartite graphs as minors. 
Together with connectivity conditions, and possibly other assumptions,
graphs with no $K_{s,t}$ as a minor can be shown to have interesting
properties relating to toughness, hamiltonicity, and other
traversability properties.  The simplest result of this kind follows
from a well-known consequence of Wagner's characterization of planar
graphs.  This consequence says that $2$-connected $K_{2,3}$-minor-free
graphs are outerplanar or $K_4$; hence, they are hamiltonian.
 For some recent examples of this type of result, involving toughness,
circumference, and spanning trees of bounded degree, see
 \cite{CEKMO11,
 	chen,
	OO12}. 

 Our work was originally motivated by trying to find forbidden minor
conditions to make $3$-connected planar graphs, or $3$-connected graphs
more generally, hamiltonian.  In examining the hamiltonicity of
$3$-connected $K_{2,4}$-minor-free graphs we were led to a complete
picture of their structure, which we then extended to
$K_{2,4}$-minor-free graphs in general.  Using this, we show in Section
\ref{consequences} that $3$-connected $K_{2,4}$-minor-free graphs are
hamiltonian, and that $2$-connected $K_{2,4}$-minor-free graphs have
hamilton paths.

 For $K_{2,4}$-minor-free graphs, or $K_{2,t}$-minor-free graphs in
general, there are a number of previous results.
 Dieng and Gavoille (see Dieng's thesis \cite{dieng}) showed
that every $2$-connected $K_{2,4}$-minor-free graph contains two vertices
whose removal leaves the graph outerplanar. Streib
and Young \cite{streib} used Dieng and Gavoille's result to show that
the dimension of the minor poset of a connected graph $G$ with no
$K_{2,4}$ minor is polynomial in $|E(G)|$.
 Chen et~al.~\cite{chen} proved that $2$-connected $K_{2,t}$-minor-free
graphs have a cycle of length at least $n/t^{t-1}$.
 Myers \cite{myers} proved that  a $K_{2,t}$-minor-free graph $G$ with
$t \geq 10^{29}$ satisfies $|E(G)| \leq (1/2)(t+1)(n-1)$; more recently
Chudnovsky, Reed and Seymour \cite{chud} showed that this is valid for
all $t \ge 2$, and provided stronger bounds for $2$-, $3$- and
$5$-connected graphs.  Our results improve their bound for $3$-connected
graphs when $t=4$.
 An unpublished paper of Ding \cite{ding2} proposes that
$K_{2,t}$-minor-free graphs can be built from slight variations of
outerplanar graphs and graphs of bounded order by adding `strips' and
`fans' using an operation that is a variant of a $2$-sum (and which
corresponds to the idea of replacing subdividable sets of edges that is
used later in this paper).
 Ding's result involves subgraphs that have $K_{2,4}$ minors, and so not
all aspects of his structure can be present in the case of
$K_{2,4}$-minor-free graphs; our results illuminate the extent to which
Ding's structure still holds.

 As part of our work we use rooted minors, where particular vertices of
$G$ must correspond to certain vertices of $H$ when we find $H$ as a
minor in $G$.
 For example, Robertson and
Seymour \cite{ix} characterized all $3$-connected graphs that have no
$K_{2,3}$ minor rooted at the three vertices on one side of the
bipartition.
 Fabila-Monroy and Wood \cite{fm-wood} characterized graphs with no
$K_4$ minor rooted at all four vertices.
 Demasi \cite{demasi} characterized all $3$-connected planar graphs with
no $K_{2,4}$ minor rooted at the four vertices on one side of the
bipartition.
 In this paper we characterize all graphs with no $K_{2,2}$ minor rooted
at two vertices on one side of the bipartition.
 This result is useful not only here, but also in the authors' proof
that $3$-connected $K_{2,5}$-minor-free planar graphs are hamiltonian
(see \cite{dissertation}).

 We begin with some definitions and notation. All graphs are simple.
 We use `$\vdel$' to denote set difference and deletion of vertices from
a graph, `$\edel$' to denote deletion of edges, `$\ctr$' to denote
contraction of edges, and `$+$' to denote both addition of edges and
join of graphs.  Since we work with simple graphs, when we contract an edge
any parallel edges formed are reduced to a single edge.

 A graph $H$ is a \textit{minor} of a graph $G$ if $H$ is
isomorphic to a graph formed from $G$ by contracting and deleting edges
of $G$ and deleting isolated vertices of $G$.  We delete multiple edges
and loops, so all minors are simple. Another way to think of a
$k$-vertex minor $H$ of $G$ is as a collection of disjoint subsets of
the vertices of $G$, $(V_1,V_2,\ldots,V_k)$ where each $V_i$ corresponds to
a vertex $v_i$ of $H$, where $G[V_i]$ (the subgraph of $G$ induced by
the vertex set $V_i$) is connected for $1 \leq i \leq k$, and for each
edge $v_iv_j \in E(H)$ there is at least one edge between $V_i$ and
$V_j$ in $G$.
 We call this a \textit{model} of $H$ in $G$.
 We will often identify minors in graphs by describing the sets
$(V_1,V_2,\ldots,V_k)$.  The set $V_i$ is known as the \textit{branch
set} of $v_i$, and may be thought of as the set of vertices in $G$ that
contracts to $v_i$ in $H$.

 Suppose we are given $S \subseteq V(G)$, $T \subseteq V(H)$, and a
bijection $f : S \to T$.  We say that a model of $H$ in $G$ is a
\textit{minor rooted at $S$ in $G$ and at $T$ in $H$ by $f$} if each $v
\in S$ belongs to the branch set of $f(v) \in T$.
 If the symmetric group on $T$ is a subgroup of the automorphism group
of $H$ (as it will be in our case) then the exact bijection $f$ between
$S$ and $T$ does not matter.

 A graph is \textit{$H$-minor-free} if it does not contain $H$ as a
minor. A \textit{$k$-separation} in a graph $G$ is a pair $(H,J)$ of
edge-disjoint subgraphs of $G$ with $G=H\cup J$, $|V(H) \cap V(J)|=k$,
$V(H)\sdel V(J) \neq \emptyset$, and $V(J)\sdel V(H) \neq \emptyset$.

 Suppose $K_{2,t}$ has bipartition $(\{a_1,a_2\}, \{b_1,b_2,\ldots,b_t\})$.
 Let $R_1$ and $R_2$ be the branch sets of $a_1$ and $a_2$ in a
model of $K_{2,t}$ in a graph $G$.  Suppose $B$ is the branch
set of $b_i$ for some $i$. Then there is a path $v_1v_2\ldots v_k$, $k \geq
3$, with $v_1 \in \ptwo_1$, $v_k \in \ptwo_2$, and $v_i \in B$ for $2 \leq i
\leq k-1$.  Let $B'=\{v_2\}$ and let $\ptwo_2'=\ptwo_2 \cup
\{v_3,\ldots,v_{k-1}\}$. We can replace $B$ with $B'$ and $\ptwo_2$ with
$\ptwo_2'$ and still have a model of $K_{2,t}$ (possibly using
fewer vertices of $G$ than before).
 Hence without loss of generality we may assume that the branch set of
each vertex $b_i$, $1 \leq i \leq t$, contains a single vertex $s_i$. Let
$S=\{s_1,s_2,\ldots,s_t\}$. We say $(\ptwo_1,\ptwo_2;S)$ represents a
\textit{standard $K_{2,t}$ minor}. Observe that $G$
contains a $K_{2,t}$ minor if and only if $G$ contains a standard
$K_{2,t}$ minor. Note that the standard model also applies to $K_{2,t}$
minors rooted at two vertices corresponding to $a_1$ and $a_2$.

 A \textit{wheel} is a graph $W_n=K_1 + C_{n-1}$ with $n \ge 4$.  A
vertex of degree $n-1$ in $W_n$ is a \textit{hub} and its incident edges
are \textit{spokes} while the remaining edges form a cycle called the \textit{rim}.  In
$W_4=K_4$ every vertex is a hub and every edge is both a spoke and a rim
edge, but in $W_n$ for $n \ge 5$ there is a unique hub and the edges are
partitioned into spokes and rim edges.
 Note that we identify wheels by their number of vertices, rather than
their number of spokes.

 A graph is \textit{outerplanar} if it has an \textit{outerplane
embedding}, an embedding in the plane with every vertex on the outer
face.
 \smallskip

In the next section, we define a class of graphs and describe several
small examples which together make up all $3$-connected
$K_{2,4}$-minor-free graphs. We begin with $3$-connected graphs because
all $4$-connected graphs on at least six vertices have a
$K_{2,4}$ minor. This is obvious for complete graphs. Otherwise, a pair
of nonadjacent vertices and the four internally disjoint paths between
them guaranteed by Menger's Theorem yield a $K_{2,4}$ minor. In Section
\ref{2connected} we extend the characterization to $2$-connected graphs. The
generalization to all graphs follows because a graph that is not
$2$-connected is $K_{2,4}$-minor-free if and only if each of its blocks
is $K_{2,4}$-minor-free.  
 Section \ref{consequences} presents applications of our
characterization to hamiltonicity, topological properties, counting, and
edge bounds.

\section{The $3$-connected case}\label{3connected}

All graphs $G$ with $|V(G)|<6$ are trivially $K_{2,4}$-minor-free; the
$3$-connected ones are $K_5$, $K_5\edel e$, $\wheelfive$, and $K_4=W_4$.
For $|V(G)| \geq 6$, first we define a class of graphs and
identify those that are $3$-connected and $K_{2,4}$-minor-free. We then
look at some small graphs that do not fit into this class. Finally, we
show that every $3$-connected $K_{2,4}$-minor-free graph is one of these
we have described.    

\subsection{A class of graphs $G_{n,r,s}^{(+)}$}
\label{ss:classG}

For $n \geq 3$ and $r, s \in \{0, 1, \ldots, n-3\}$, let $G_{n,r,s}$
consist of a spanning path $v_1v_2\ldots v_n$, which we call the
\textit{spine}, and edges $v_1v_{n-i}$ for $1 \leq i \leq r$ and
$v_nv_{1+j}$ for $1 \leq j \leq s$. The graph $G_{n,r,s}^+$ is
$G_{n,r,s}+v_1v_n$; we call $v_1v_n$ the \textit{plus edge}.
 All graphs $G^{(+)}_{n,r,s}$ are planar.
 The graph $G^+_{n,1,n-3}$ is a wheel $W_n$ with hub $v_n$.
 Examples are shown in Figure~\ref{fig:G_p,r,sm}.
 Since $G_{n,r,s}^{(+)} \isom G_{n,s,r}^{(+)}$ we often assume $r \leq
s$. 

\begin{figure}
	\centering \scalebox{.8}{\includegraphics{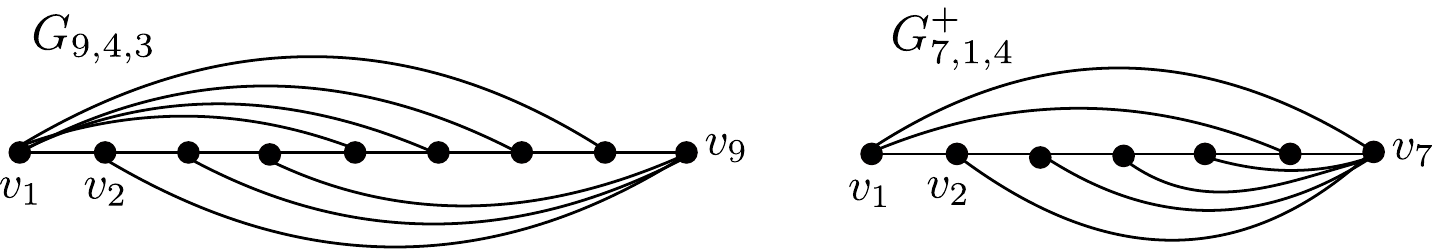}}
	\captionof{figure}{}
	\label{fig:G_p,r,sm} 
\end{figure}  

In the following three lemmas we first determine when a graph
$G_{n,r,s}^{(+)}$ is $3$-connected, and then when it is
$K_{2,4}$-minor-free.

\begin{lemma} For $n \geq 4$, $G = G_{n,r,s}^{(+)}$ is $3$-connected if and
only if (i) $r=1$, $s = n-3$, and the plus edge is present (or
symmetrically $s=1$, $r = n-3$, and the plus edge is present) or (ii)
$r,s \geq 2$ and $r+s \geq n-2$. 
\label{lem:G3conn}
\end{lemma}

\begin{proof} Assume that $r \le s$.
 To prove the forward direction, assume $G$ is $3$-connected and first
suppose $r=1$. If the plus edge is not present, then $v_1$ has degree
$2$ and $\{v_2,v_{n-1}\}$ is a $2$-cut. Similarly if $s \leq n-4$, then
$v_{n-2}$ has degree $2$ and $\{v_{n-3},v_{n-1}\}$ is a $2$-cut. Next
suppose $r,s \geq 2$. If $r+s \leq n-3$, then there is necessarily a
degree $2$ vertex $v_i$ with $4 \leq i \leq n-3$ and hence a $2$-cut in
$G$.

 To prove the reverse direction, assume (i) or (ii).  If (i) holds, $G$
is a wheel, which is $3$-connected, so we may assume that (ii) holds.
 To show $3$-connectedness we find three internally disjoint paths
between each possible pair of vertices. For $v_1$ and $v_n$ we have
paths $v_1v_2v_n$, $v_1v_{n-1}v_n$, and
$v_1v_{n-2}v_{n-3}...v_{n-r}v_{1+s}v_n$ (where possibly
$v_{n-r}=v_{1+s}$). Next suppose that only one of $v_1$ and $v_n$ is in
the considered pair, say $v_1$ without loss of generality. First
consider $v_1$ and $v_i$ where $n-r \leq i \leq n-1$. When $v_1v_{i+1}
\in E(G)$, then the three disjoint paths are $v_1v_2...v_i$, $v_1v_i$,
and $v_1v_{i+1}v_i$. When $v_1v_{i+1} \notin E(G)$, then $v_1v_{i-1} \in
E(G)$ and $v_{i-1} \neq v_2$ and the three disjoint paths are
$v_1v_2v_nv_{n-1}...v_i$, $v_1v_{i-1}v_i$, and $v_1v_i$. Now consider
$v_1$ and $v_i$ where $2 \leq i \leq n-r-1$. Then the three disjoint
paths are $v_1v_2...v_i$, $v_1v_{n-r}v_{n-r-1}...v_i$, and
$v_1v_{n-r+1}v_{n-r+2}...v_nv_i$.
 Finally consider $v_i$ and $v_j$ where $i < j$ and $i,j \neq 1,n$. If
$v_i$ and $v_j$ are both adjacent to the same end vertex, say $v_1$,
where $i,j \neq 2$, then the three disjoint paths are
$v_iv_{i+1}...v_j$, $v_iv_1v_j$, and $v_iv_{i-1}...v_2v_nv_{n-1}...v_j$.
Otherwise the three disjoint paths are $v_iv_{i+1}...v_j$,
$v_iv_{i-1}...v_1v_j$, and $v_jv_{j+1}...v_nv_i$.
 \end{proof}

\begin{lemma} For $n \geq 6$, $G=G_{n,r,s}^{(+)}$ is $K_{2,4}$-minor-free if and only if $r+s \leq n-1$. 
\label{lem:GK24}
\end{lemma}

 \begin{proof}
 To prove the forward direction, suppose $r+s \geq n$. Then there are
vertices $v_i$ and $v_{i+1}$ such that both $v_1$ and $v_n$ are adjacent
to both $v_i$ and $v_{i+1}$ and $3 \leq i \leq n-3$. Then there is a
standard $K_{2,4}$ minor $(\ptwo_1,\ptwo_2;S)$ in $G$: let
$S=\{v_2,v_i,v_{i+1},v_{n-1}\}$, $\ptwo_1=\{v_1\}$, and $\ptwo_2=\{v_n\}$. 

Now suppose that $r+s \leq n-1$. We claim that if $G$ has a standard $K_{2,4}$ minor $(\ptwo_1,\ptwo_2;S)$, then $v_1 \in \ptwo_1$ and $v_n \in \ptwo_2$ (or vice versa). The graph $G\vdel v_1$ is outerplanar and thus has no $K_{2,3}$
minor. Therefore, if $G$ has a $K_{2,4}$ minor, then it must include $v_1$.
 We cannot have $v_1 \in S$ because then the outerplanar graph
$G\vdel v_1$ would have a $K_{2,3}$ minor. 
 By symmetry, $v_n$ must also be included in the
minor and $v_n \notin S$. 
 If $v_1,v_n \in \ptwo_i$, then $G\vdel\{v_1,v_n\}$ has a $K_{1,4}$
minor, but $G\vdel\{v_1,v_n\}$ is a path and there is no $K_{1,4}$ minor in
a path. The only remaining possibility is $v_1 \in \ptwo_1$ and $v_n \in
\ptwo_2$ (or vice versa).

 Let $N(v)$ denote the set of neighbors of $v$.
 Let $A = N(v_1)\sdel\{v_2\} = \{v_{n-r},v_{n-r+1},\ldots,v_{n-1}\}$ and $B
= N(v_n)\sdel\{v_{n-1}\} = \{v_2,v_3,\ldots,v_{s+1}\}$, which intersect
only if $v_{n-r}=v_{s+1}$. Suppose $G$ has a standard $K_{2,4}$ minor
$(\ptwo_1,\ptwo_2;S)$. Then by the claim proved in the previous paragraph, $v_1 \in \ptwo_1$
and $v_n \in \ptwo_2$. We consider the makeup of $S$. Suppose
$\{s_1,s_2,s_3\} \subseteq S \cap A$, in that order along the spine.
Since $\{v_1,s_1,s_3\}\subseteq \ptwo_1 \cup \{s_1,s_3\}$ separates
$s_2$ and $v_n$, and $v_n \in \ptwo_2$, we cannot have $\ptwo_2$
adjacent to $s_2$, which is a contradiction. Thus $|S \cap A| \leq 2$.
Symmetrically, $|S \cap B| \leq 2$. We must have $s_1,s_2 \in S \cap A$
and $s_3, s_4 \in S \cap B$ in the order $s_4,s_3,s_2,s_1$ along the
spine. Since $v_n \in \ptwo_2$, there must be a $v_n s_2$-path in
$G\vdel\{v_1,s_1,s_3,s_4\}$, and hence $s_3 \ne v_{s+1}$. Then
$v_{s+1}$ is a cutvertex separating $v_n$ and $s_2$ in
$G\vdel\{v_1,s_1,s_3,s_4\}$, so $v_{s+1} \in \ptwo_2$. Now there
must also be a $v_1 s_3$-path in $G\vdel\{v_n,v_{s+1},s_4\}$ but no such
path exists. Thus there is no $K_{2,4}$ minor.  \end{proof}

 Define $\G$ to be the set of (labeled) graphs of the form
$G^{(+)}_{n,r,s}$ that are both $3$-connected and $K_{2,4}$-minor-free.
 Of the four $3$-connected graphs on fewer than six vertices, three are
planar, and all three belong to $\G$: $K_5\edel e \isom G_{5,2,2}^+$,
$\wheelfive \isom G_{5,1,2}^+ \isom G_{5,2,2}$, and $K_4=W_4 \isom
G_{4,1,1}^+$.  From this and Lemmas~\ref{lem:G3conn} and~\ref{lem:GK24}
we get
 $$\G=\{G_{n,1,n-3}^+, G_{n,n-3,1}^+: n \geq 4\} \cup
  \{G_{n,r,s}^{(+)}: n \geq 5,\;
	r, s \in \{2, 3, \ldots, n-3\},\;
	r+s=n-1 \text{ or } n-2\}.$$  
 Let $\Gi$ denote the class of all graphs isomorphic to a graph in $\G$. Note that graphs in $\G$ are $3$-sums of two wheels, a fact we will see in more detail later on.

There are some isomorphisms between graphs in $\G$ and also
symmetries within certain graphs of the class.
 Let $\rho=\rho_n$ be the involution with $\rho(v_i) = v_{n+1-i}$
for $1 \le i \le n$.  Then $\rho$ provides the isomorphism (in both
directions) between $G^{(+)}_{n,r,s}$ and $G^{(+)}_{n,s,r}$ that we have
already noted; if $r=s$ it is an automorphism.
 The graph $G_{n,1,n-3}^+$ is isomorphic to $W_n$, with $v_n$ as a hub. 
It has the obvious symmetries.

\begin{figure}[h]
	\centering \scalebox{.7}{\includegraphics{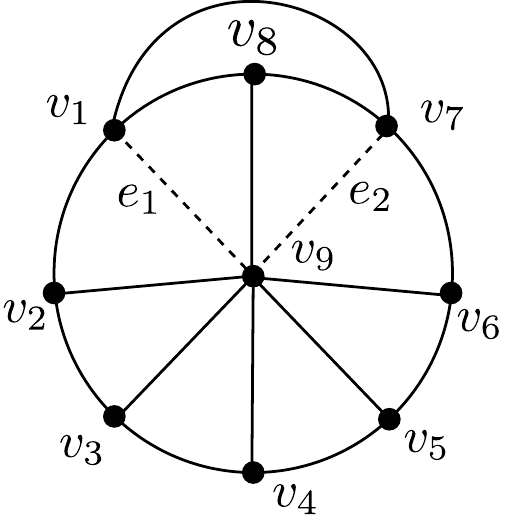}}
	\captionof{figure}{}
	\label{fig:isomorph} 
\end{figure}

 Define $\sigma=\sigma_n$ to be the involution fixing $v_{n-1}$ and
$v_n$ and with $\sigma(v_i)=v_{n-1-i}$ for $1 \leq i \leq n-2$.  Then
$\sigma$ is an automorphism of $G_{n,2,n-4}$, an isomorphism (in both
directions) between $G_{n,2,n-4}^+$ and $G_{n,2,n-3}$, and an
automorphism of $G_{n,2,n-3}^+$.
 The case $n=9$ is illustrated in Figure~\ref{fig:isomorph}, where
$\sigma=\sigma_9$ corresponds to reflection about a vertical axis.
 The graph without the dashed edges $e_1$ and $e_2$ is $G_{9,2,5}$.
 With the edge $e_1$, the graph is $G_{9,2,5}^+$ and with $e_2$, the
graph is $G_{9,2,6}$.
 With both edges $e_1$ and $e_2$, the graph is $G_{9,2,6}^+$.
 In general $\sigma$ maps the spine $P=v_1v_2\ldots v_n$ to the path
$\sigma(P)=v_{n-2}v_{n-3}\ldots v_2v_1v_{n-1}v_n$.
 For $G^{(+)}_{n,r,s}$ with $r=2$ we call $\sigma(P)$ the \textit{second
spine}.
 When $s=2$ we have a similar involution $\sigma'$, and the path
$\sigma'(P) = v_1 v_2 v_n v_{n-1} \ldots v_4 v_3$ can be regarded as an
extra spine.  When $r=s=2$, $\sigma'(P)$ is the image of $\sigma(P)$
under the automorphism $\rho$.

 Finally, besides some obvious special symmetries when $n=4$ or $5$,
$G_{6,2,2}$ is vertex-transitive and is isomorphic to the triangular
prism.

These symmetries and isomorphisms will be important later, particularly
in Section 3 when we discuss which edges of $G \in \G$ can be
subdivided without creating a $K_{2,4}$ minor.  Up to isomorphism the class $\G$
contains one $4$-vertex graph and $2n-8$ $n$-vertex graphs
for each $n \geq 5$.

 We now examine the effect of deleting or contracting a single edge of a
graph in $\G$.

 \begin{lemma}\label{lem:Gdeledge}
 Suppose $G = G^{(+)}_{n,r,s} \in \G$ and $e \in E(G)$.  The
following are equivalent.

 \hangitem (i) $G\edel e \in \G$.

 \hangitem (ii) $G\edel e$ is $3$-connected.

 \hangitem (iii) $G$ is not a wheel and either $e$ is a plus edge, or
$r+s=n-1$ and $e \in \{v_1 v_{n-r}, v_n v_{1+s}\}$.
 \end{lemma}

 \begin{proof}
 Clearly (iii) $\implies$ (i) $\implies$ (ii).  If (iii) does not hold
then $G\edel e$ has at least one vertex of degree $2$, so (ii) does not hold;
thus (ii) $\implies$ (iii).
 \end{proof}

\begin{table}[h]
 \begin{center}
 \caption{\label{tbl:contract}:
	Contracting an edge $e$ in $G=G^{(+)}_{n,r,s}$ with $r, s \ge 2$}
 \begin{tabular}{|l|l|c|c|}
 \hline
 \vrule height13pt depth0pt width0pt 
 $e$ & $G \ctr e$ isomorphic to & $G\ctr e$ is $3$-conn.?
					& $G\ctr e \in \Gi$? \\
 \hline
 spine edges & & & \\
 \quad\llap{$*$ }$v_{1+s}v_{n-r}$, \quad $r+s=n-2$
		& $G^{(+)}_{n-1,r,s}$ & yes & yes \\[1pt]
 \quad $v_{n-i} v_{n-i+1}$, \quad $2 \le i \le r$, $r \ge 3$
		& $G^{(+)}_{n-1,r-1,s}$ & yes & yes \\[1pt]
 \quad $v_{n-2}v_{n-1}$, \quad $r=2$
		& $G^{(+)}_{n-1,1,n-4}$ & if plus edge &
			if plus edge \\[1pt]
 \quad\llap{$*$ }$v_{n-1} v_n$, \quad $r \ge 3$
		& $G^+_{n-1,r-1,s}$ & yes & yes \\[1pt]
 \quad\llap{$*$ }$v_{n-1} v_n$, \quad $r=2$
		& $G^+_{n-1,1,n-4} \isom W_{n-1}$ & yes & yes \\
 non-spine edges & & & \\
 \quad $v_1v_n$ (plus edge)
		& $K_1 + P_{n-2}$ & no & no \\[1pt]
 \quad $v_1v_{n-i}$, \quad $2 \le i \le r-1$
		& $G^{(+)}_{n-1,r-1,s} \edel v_{n-i-1}v_{n-i}$ & no & no \\[1pt]
 \quad $v_1v_{n-1}$, \quad $r \ge 3$
		& $G^+_{n-1,r-1,s} \edel v_{n-2}v_{n-1}$ & no & no \\
 \qquad or $r=2$ and $s=n-4$ & & & \\
 \quad $v_1v_{n-1}$, \quad $r = 2$ and $s=n-3$
		& $G^+_{n-1,1,n-4} \isom W_{n-1}$ & yes & yes \\[1pt]
 \quad $v_1 v_{n-r}$, \quad $r+s=n-2$
		& $G^{(+)}_{n-1,r,s} \edel v_{n-r-1}v_{n-r}$ & no & no \\[1pt]
 \quad $v_1 v_{n-r}$, \quad $r+s=n-1$, $s \ge 3$
		& $G^+_{n-1,r,s-1} \edel v_{n-r-1}v_{n-r}$ & no & no \\[1pt]
 \quad $v_1 v_{n-r}$, \quad $r+s=n-1$, $s = 2$
		& $G^+_{n-1,n-4,1} \edel v_2v_3$ & no & no \\[1pt]
 \hline
 \end{tabular}
 \end{center}
\end{table}

 Now consider contracting an edge $e$ of $G=G^{(+)}_{n,r,s}$.  If $n=4$
then $G=K_4$ and $G\ctr e \isom K_3$ for any edge $e$, so assume that $n
\ge 5$.  If $G$ is a wheel $W_n$ then we obtain $W_{n-1}$ if we contract
a rim edge, and $K_1 + P_{n-2}$ if we contract a spoke.  Therefore
assume $G$ is not a wheel, so $r, s \ge 2$.
 The effects of contracting edges in this case are shown in Table
\ref{tbl:contract}.  Here the superscript `${}^{(+)}$' means that the
plus edge is present in $G\ctr e$ precisely if it is present in $G$.
 Edges not included in the table are
covered by the symmetry $\rho$ that swaps $r$ and $s$, $v_i$ and $v_{n+1-i}$.
We may summarize the results as follows.

 \begin{lemma}\label{lem:Gcontredge}
 Suppose $G = G^{(+)}_{n,r,s} \in \G$ and $e \in E(G)$.

 \hangitem (i) If $n \ge 5$
then $G\ctr e$ is isomorphic to a graph in $\G$ with at most one edge
deleted.

 \hangitem (ii) If $n \ge 4$ then $G\ctr e \in \Gi$ if and only if $G\ctr e$
is $3$-connected.

 \hangitem (iii) If $G$ is a wheel with $n \ge 5$ then some $G \ctr e$
is isomorphic to $W_{n-1}$, and every $G \ctr e \in \Gi$ is isomorphic
to $W_{n-1}$.
 If $G$ is not a wheel then (from the starred entries in Table
\ref{tbl:contract}) some $G\ctr e$ is isomorphic to each of
 $G^+_{n-1,r-1,\min(s,n-4)}$,
 $G^+_{n-1,\min(r,n-4),s-1}$
 and, if $r+s=n-2$, also $G^{(+)}_{n-1,r,s}$;
 and any $G\ctr e \in \Gi$ is isomorphic to a spanning subgraph of one
of these.
 \end{lemma}

 Now we apply these results to the structure of minors of graphs in $\G$
or $\Gi$.

\begin{corollary}\label{lem:minorclosed}
 Every minor of a graph in $\Gi$ is a subgraph of some graph in $\Gi$.
 \end{corollary}

 \begin{proof}
  Apply Lemma \ref{lem:Gcontredge}(i) repeatedly to replace contractions
by deletions (details are left to the reader).
 \end{proof}


 \begin{lemma}\label{lem:splitter}
 If a $3$-connected graph $H$ is a minor of a $3$-connected graph $G$,
then there is a sequence of $3$-connected graphs $G_0, G_1, \ldots, G_k$
where $G_0 \isom G$, $G_k \isom H$, and each $G_{i+1}$ is obtained from
$G_i$ by contraction or deletion of a single edge.
 \end{lemma}

 \begin{proof}
 Seymour's Splitter Theorem \cite{seymour} as applied to graphs, or a
similar result of Negami \cite{negami}, says that our result is true if
$H$ is not a wheel, or if $H$ is the largest wheel minor of $G$.
 Seymour's operations and connectivity are defined for graphs with loops
and multiple edges, not simple graphs, which is why a sequence of minors
$G_0, G_1, \ldots, G_i$ cannot be continued to reduce a large wheel
minor $G_i$ to a smaller one.  In particular, contracting a rim edge of
a wheel in his definition yields a pair of parallel edges and so by his
definition the graph is not $3$-connected.  With our definition, where
we reduce parallel edges to a single edge after contraction, we can
contract a rim edge of a wheel $W_{\ell}$, $\ell \ge 5$, to obtain the
smaller wheel $W_{\ell-1}$, which is still $3$-connected.  Therefore, we
can continue the sequence of operations to also reach wheel minors $H$
that are not the largest wheel minor.
 \end{proof}

 \begin{corollary}\label{cor:minorclosed}
 If $H$ is a $3$-connected minor of
$G \in \Gi$ then $H \in \Gi$.
 \end{corollary}

 \begin{proof} Take the $3$-connected sequence $G \isom G_0, G_1,
\ldots, G_k \isom H$ given by Lemma \ref{lem:splitter}.  From Lemmas
\ref{lem:Gdeledge} and \ref{lem:Gcontredge}(ii), if $G_i \in \Gi$ then
$G_{i+1} \in \Gi$ also, and the result follows by induction.
 \end{proof}

\subsection{Small cases}
\label{ss:smallcases}

 Figure~\ref{fig:smallcasesnew} shows nine small graphs that are
$3$-connected (easily checked), not in $\G$ (also easily checked; all
but $D$ have a $K_{3,3}$ minor and so are nonplanar) and
$K_{2,4}$-minor-free.
 The first graph, $K_5$, is the only $3$-connected
graph on fewer than six vertices that is not in $\G$.
 To prove that the other eight graphs are $K_{2,4}$-minor-free we
examine the two maximal graphs $C^+$ and $D$, and show that the rest are
minors of $C^+$.

\begin{figure}[h]
\centering \scalebox{.7}{\includegraphics{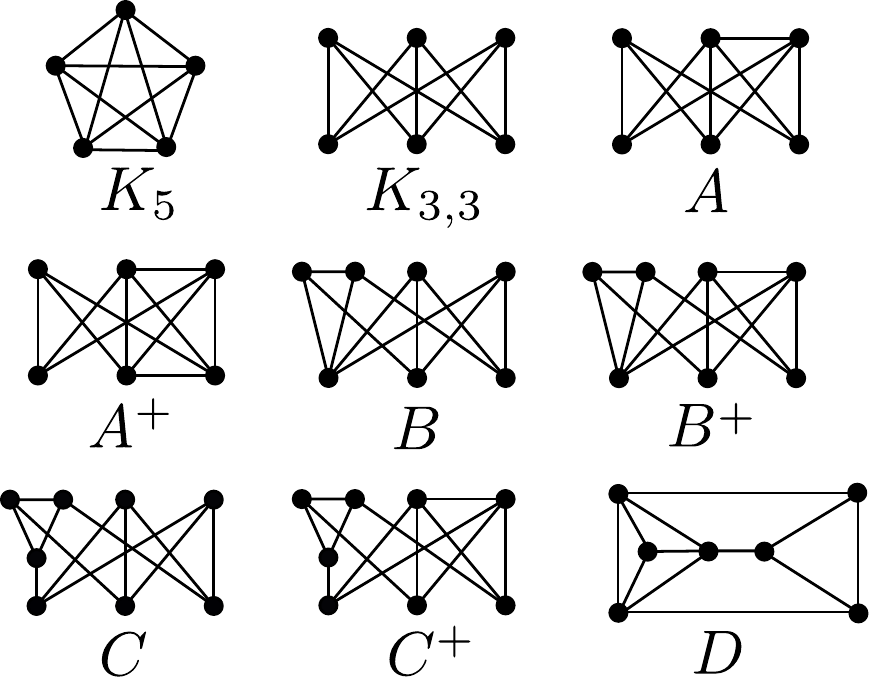}}
\caption{\label{fig:smallcasesnew}}
\end{figure}

\begin{lemma}  The graph $C^+$ is $K_{2,4}$-minor-free.
\label{lem:scC+}
\end{lemma}

\begin{proof} Consider $C^+$ with vertices labeled as on the left in
Figure~\ref{fig:smallcasesC+m}.  Suppose there is a standard $K_{2,4}$ minor
$(\ptwo_1,\ptwo_2;S)$ in $C^+$ and suppose $|\ptwo_1|=1$. Then $\ptwo_1$ must be
either $v_4$ or $v_5$ since these are the only vertices of degree $4$.
Say, without loss of generality, $\ptwo_1 = \{v_4\}$. Then $S= \{v_5,
v_6,v_7,v_8\}$, and $\ptwo_2$ must be a subset of $\{v_1,v_2,v_3\}$. None
of these three vertices are adjacent to $v_5$, however, so we cannot
have $\ptwo_2$ adjacent to $v_5$ and thus we cannot have $|\ptwo_1| = 1$, or
symmetrically $|\ptwo_2| = 1$.  Thus $|\ptwo_1| \geq 2$ and $|\ptwo_2| \geq 2$
and since $|V(C^+)|=8$, $|\ptwo_1|=|\ptwo_2|=2$. 

Let $T$ be a triangle with a set $N$ of neighbors with $|N|=3$. Suppose $\ptwo_1 \subseteq V(T)$. Then we would have $N \subseteq S$ along with the third vertex $t$ of $T$, but $N$ separates $t$ from the rest of the graph so $\ptwo_2$ cannot be adjacent to $t$. Thus $\ptwo_1$ (or symmetrically $\ptwo_2$) cannot consist of two vertices in a triangle with only three neighbors. In $C^+$, we have the following triples of vertices which form such triangles: $\{v_1,v_2,v_3\}$, $\{v_4,v_5,v_6\}$, $\{v_4,v_5,v_7\}$, and $\{v_4,v_5,v_8\}$. The only remaining pairs of adjacent vertices that could make up $\ptwo_1$ or $\ptwo_2$ are $\{v_3,v_6\}$, $\{v_2,v_8\}$, and $\{v_1,v_7\}$ where all three cases are symmetric. If $\ptwo_1=\{v_3,v_6\}$, then $\ptwo_2$ must be $\{v_7,v_8\}$ but this set is not an option for $\ptwo_2$. \end{proof}

\begin{figure}[h]
	\centering \scalebox{.8}{\includegraphics{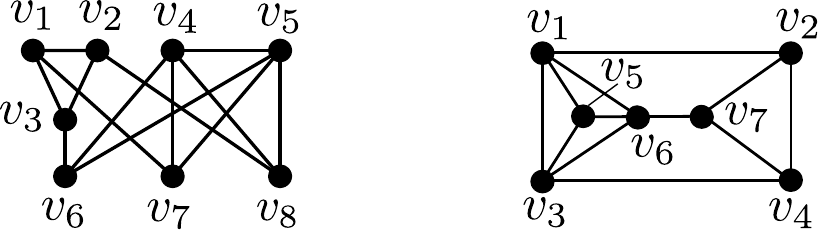}}
	\captionof{figure}{}
	\label{fig:smallcasesC+m} 
\end{figure}  

\begin{lemma}  The graph $D$ is $K_{2,4}$-minor-free.
\label{lem:scD}
\end{lemma}

\begin{proof}
 It is easy to check that $D$ has no subgraph isomorphic to $K_{2,4}$,
nor does $D\ctr e$ for any $e \in E(D)$.
 Hence $D$ is $K_{2,4}$-minor-free since $|V(D)| = 7$. \end{proof}

\begin{figure}[h]
 \centering
  \scalebox{.7}{\includegraphics{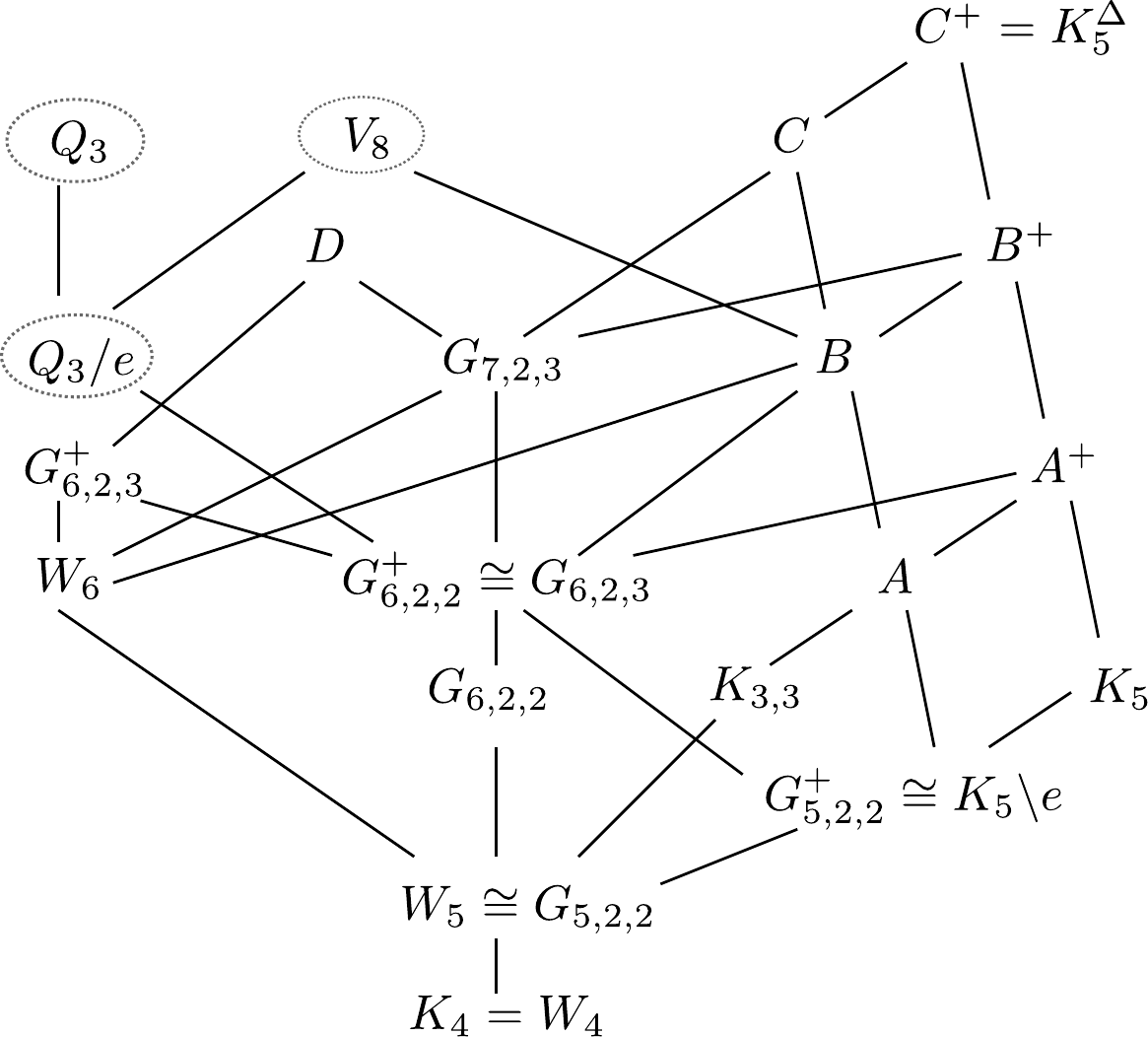}}
  \caption{: $Q_3$, $V_8$, $C^+=K_5^{\Delta}$, $D$, and their
	$3$-connected minors
 \label{fig:minordiagram}}
\end{figure} 

 Figure \ref{fig:minordiagram} shows what we will prove is the Hasse
diagram for the minor ordering of all $3$-connected minors of $C^+$
(also labeled $K_5^\Delta$, following Ding and Liu \cite{ding}) and $D$.
 For future reference the figure also includes three additional, circled
graphs $Q_3$ (the cube), $Q_3/e$ (contract any edge of $Q_3$) and $V_8$
(the $8$-vertex twisted cube or M\"{o}bius ladder).
 Unlike the other graphs, these three have $K_{2,4}$ minors, as shown by
the minor in $Q_3/e$ on the right in Figure \ref{fig:small_cases_B+m}.
 Here, and later, a $K_{2,4}$ minor is indicated by two groups of
vertices circled by dotted curves representing the two vertices in one
part of the bipartition of $K_{2,4}$, and four triangular vertices
representing the other part.

\begin{figure}[h]
\centering \scalebox{.7}{\includegraphics{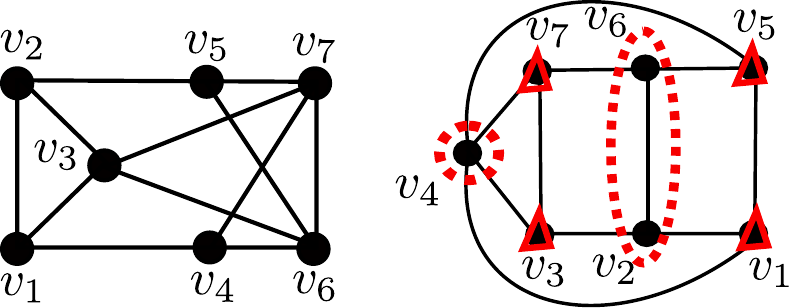}}
\captionof{figure}{}
\label{fig:small_cases_B+m}
\end{figure}

\begin{lemma}
 Figure~\ref{fig:minordiagram} is the Hasse diagram for all
$3$-connected minors (up to isomorphism) of $C^+$, $D$, $Q_3$ and $V_8$.
 \label{lem:allminors}
\end{lemma}

\begin{proof}
 By Lemma \ref{lem:splitter}, we can proceed by single edge deletions
and contractions, and we do not need to consider further minors once we
reach a graph that is not $3$-connected.
 The figure is clearly correct for the $3$-connected graphs on four or
five vertices, so we consider only graphs with at least six vertices.
 Also, the $3$-connected minors for graphs in $\G$ follow from Lemmas
\ref{lem:Gdeledge} and \ref{lem:Gcontredge}(iii), so we consider only
graphs not in $\G$.

 In what follows results of all deletions or contractions are identified
only up to isomorphism.  When we lose $3$-connectivity, in all but one
case there will be at least one vertex of degree $2$.  We work upwards
in the figure.

 For the graphs $K_{3,3}$, $A$ and $A^+$
label the vertices consecutively along the top row then the bottom row
in Figure \ref{fig:smallcasesnew}.
 For $K_{3,3}$, deleting any edge loses $3$-connectivity; contracting
any edge results in $W_5$.  For $A$, deleting $v_2v_3$ yields $K_{3,3}$,
and deleting any other edge loses $3$-connectivity.  Contracting an edge
incident with $v_1$ yields $K_5\edel e$, contracting $v_2 v_3$ loses
$3$-connectivity, and contracting any other edge incident with $v_2$ or
$v_3$ yields $W_5$.  For $A^+$, all edges are equivalent up to symmetry
to one of $v_1v_4$, $v_1v_5$, $v_2v_3$ or $v_2v_5$.  Deleting $v_1v_4$
or $v_1v_5$ loses $3$-connectivity, deleting $v_2 v_3$ gives $A$, and
deleting $v_2v_5$ gives $G_{6,2,3}$.  Contracting $v_1v_4$ gives $K_5$,
contracting $v_1v_5$ gives $K_5\edel e$, contracting $v_2 v_3$ loses
$3$-connectivity, and contracting $v_2 v_5$ yields $W_5$.

 For $B$ and $B^+$ we redraw $B^+$ as on the left in
Figure~\ref{fig:small_cases_B+m} and take $B = B^+ \edel v_6v_7$.
 For $B$, deleting any edge loses $3$-connectivity. Up to symmetry,
there are five edge contractions to consider:
 $v_1v_2$, $v_1v_3$, $v_1v_4$, $v_3v_6$ and $v_4v_6$.
 Contracting $v_1v_2$ yields $K_{3,3}$,
 contracting $v_1v_3$ loses $3$-connectivity, 
 contracting $v_1v_4$ yields $A$,
 contracting $v_3v_6$ results in $W_6$, and
 contracting $v_4v_6$ gives $G_{6,2,3}$.
 For $B^+$, all edges are equivalent up to symmetry to six possibilities:
 $v_1v_2$, $v_1v_3$, $v_1v_4$, $v_3v_6$, $v_4v_6$ and $v_6v_7$.
 Deleting $v_1 v_2$, $v_1 v_3$, $v_1 v_4$ or $v_4 v_6$ loses
$3$-connectivity,
 deleting $v_3v_6$ yields $G_{7,2,3}$, and
 deleting $v_6 v_7$ results in $B$.
 Contracting $v_1v_3$ or $v_6v_7$ loses $3$-connectivity,
 contracting $v_1 v_2$ results in $A$,
 contracting $v_1 v_4$ yields $A^+$,
 contracting $v_3 v_6$ gives $W_6$, and
 contracting $v_4 v_6$ gives $G_{6,2,3}$.

 For $C^+$ we label the vertices as on the left in Figure
\ref{fig:smallcasesC+m} and take $C = C^+ \edel v_4v_5$.
 Deleting any edge of $C$ loses $3$-connectivity.
 Up to symmetry, there are three edge contractions to consider:
$v_1v_2$, $v_1v_7$ and $v_4v_6$.
 Contracting $v_1v_2$ loses $3$-connectivity,
 contracting $v_1v_7$ results in $B$, and
 contracting $v_4v_6$ yields $G_{7,2,3}$.
 For $C^+$, deleting $v_4v_5$ yields $C$ and deleting any other edge
loses $3$-connectivity.
 Up to symmetry, there are four edge contractions of $C^+$ to consider:
$v_1v_2$, $v_1 v_7$, $v_4 v_5$ and $v_4 v_6$.
 Contracting $v_1 v_2$ or $v_4 v_5$ loses $3$-connectivity,
 contracting $v_1 v_7$ results in $B^+$, and
 contracting $v_4 v_6$ gives $G_{7,2,3}$.

 We label $D$ as on the right in
Figure~\ref{fig:smallcasesC+m}.
 Up to symmetry all edges are equivalent to one of four edges: $v_1
v_3$, $v_2v_4$, $v_5 v_6$ and $v_6 v_7$.
 Deleting $v_1 v_3$ results in $G_{7,2,3}$, and deleting any of the
other three edges loses $3$-connectivity.
 Contracting $v_1 v_3$ or $v_2 v_4$ loses $3$-connectivity,
 contracting $v_5 v_6$ yields the triangular prism $G_{6,2,2}$, and
 contracting $v_6 v_7$ results in $G^+_{6,2,3}$.

 Finally, consider $Q_3/e$, $Q_3$ and $V_8$.
 Label $Q_3/e$ as shown on the right in Figure~\ref{fig:small_cases_B+m}.
 Every edge in $Q_3/e$ is adjacent to a degree $3$ vertex so
deleting any edge loses $3$-connectivity.
 Up to symmetry, there are four edge contractions to consider:
$v_1v_2$, $v_3v_4$, $v_2v_6$ and $v_3v_7$.
 Contracting $v_3v_4$ loses $3$-connectivity, and contracting $v_2 v_6$
also loses $3$-connectivity (without creating a vertex of degree $2$).
 Contracting $v_1v_2$ results in $G_{6,2,3}$, and
 contracting $v_3v_7$ yields $G_{6,2,2}$.
 In the cube $Q_3$ all edges are symmetric; deleting any edge loses
$3$-connectivity, and contracting any edge yields $Q_3/e$.
 We may take $V_8$ to be $C_8 = (v_1 v_2 \ldots
v_8)$ with added diagonals $v_i v_{i+4}$ for $1 \le i \le 4$.  Deleting
any edge loses $3$-connectivity, contracting a $C_8$ edge results in
$B$, and contracting a diagonal yields $Q_3/e$.
 \end{proof}
 
Considering the minors of $C^+$, we obtain the following.

\begin{corollary} 
 The graphs $C$, $B^+$, $B$, $A^+$, $A$, and $K_{3,3}$ are
$K_{2,4}$-minor-free.  
 \label{cor:minorsC+}
 \end{corollary}

\subsection{Characterization of $3$-connected graphs}

 \begin{theorem} Let $G$ be a $3$-connected graph. Then $G$ is
$K_{2,4}$-minor-free if and only if $G \in \Gi$ or $G$ is isomorphic to
one of the nine small exceptions shown in
Figure~\ref{fig:smallcasesnew}.
 \label{thm:main24}
 \end{theorem}

 Our original proof of this theorem examined the structure of a
$3$-connected $K_{2,4}$-minor-free graph relative to a longest
non-hamilton cycle in the graph.  We analyzed cases and either derived a
contradiction with a longer non-hamilton cycle or a $K_{2,4}$ minor, or
found a desired graph.
 However, we then discovered the recent systematic investigation by Ding
and Liu \cite{ding}, characterizing $H$-minor-free graphs for all
$3$-connected graphs $H$ on at most eleven edges.  These allow
us to give a shorter proof, which we present here.

 First we give some definitions. Denote by Oct$\edel e$ the graph
obtained from the octahedron by removing one edge. A \textit{$3$-sum} of
two $3$-connected graphs $G_1$ and $G_2$ is a graph $G$ obtained by
identifying a triangle of $G_1$ with a triangle of $G_2$ and possibly
deleting some of the edges of the common triangle as long as no degree
$2$ vertices are created. Any $2$-cut in $G$ would lead to a $2$-cut in
either $G_1$ or $G_2$ so $G$ is $3$-connected. An example is the graph
$C^+$ which is a $3$-sum of $K_5$ and a triangular prism. A
\textit{common $3$-sum} of three or more graphs is formed by specifying
one triangle in each graph and identifying all as a single triangle
called the \textit{common triangle}; again edges of the common triangle
may be deleted as long as no degree $2$ vertices are created. Let
$\mathcal{S}$ be the set of all graphs formed by taking common $3$-sums
of wheels and triangular prisms.  All graphs in $\mathcal{S}$
are $3$-connected. We use the following result due to Ding and Liu.

 \begin{theorem}[Ding and Liu \cite{ding}] Up to isomorphism the family
of $3$-connected Oct$\edel e$-minor-free graphs consists of graphs in
$\mathcal{S}$ and $3$-connected minors of $V_8$, $Q_3$, and $C^+$.
 \label{thm:Oct-e}
 \end{theorem}

 \begin{proof} [Proof of Theorem~\ref{thm:main24}]
 The results of subsections \ref{ss:classG} and \ref{ss:smallcases} give
the reverse direction of the proof.

 For the forward direction, Oct$\edel e$ contains $K_{2,4}$ as a
subgraph, so all $3$-connected $K_{2,4}$-minor-free graphs must be
Oct$\edel e$-minor-free graphs as described in Theorem~\ref{thm:Oct-e}. 
 We must decide which of those graphs are actually
$K_{2,4}$-minor-free.
 By Lemma~\ref{lem:allminors}, Figure \ref{fig:minordiagram} gives all
$3$-connected minors of $V_8$, $Q_3$, and $C^+$ up to isomorphism. 
The $K_{2,4}$-minor-free ones are uncircled; all are in $\G$ or
one of the nine small exceptions.

 So we must determine which members of $\mathcal{S}$ are
$K_{2,4}$-minor-free. Any common $3$-sum of four or more graphs has a
$K_{3,4}$ minor (the three vertices of the common triangle form the part
of size three) and hence a $K_{2,4}$ minor.  Thus, we consider common
$3$-sums of at most three graphs, analyzed according to the numbers of
wheels and prisms.

\begin{figure}[h]
	\centering \scalebox{.9}{\includegraphics{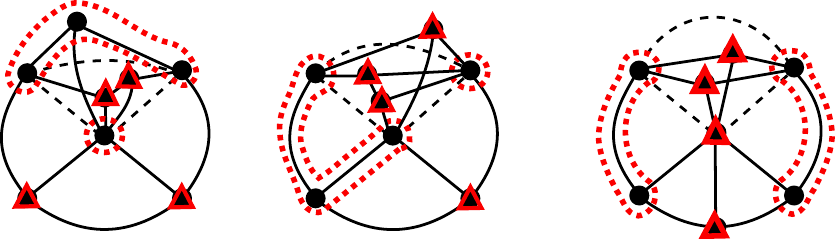}}
	\captionof{figure}{}
	\label{fig:S_1m}
\end{figure}

First consider a common $3$-sum of three wheels, $W_k$, $W_{\ell}$, and
$W_m$. For $k=\ell=5$ and $m=4$, since all vertices of $W_4=K_4$ are
equivalent, there are two ways up to symmetry to form a common $3$-sum
(disregarding the possible existence of the edges of the common
triangle): the hubs of the two wheels are either identified or not. 
Both result in a $K_{2,4}$ minor, as shown in the left and middle
pictures of Figure~\ref{fig:S_1m}. The dashed edges are the edges of the
common triangle which may or may not be present in the common $3$-sum.
Since graphs with $k, \ell \geq 5$ and $m \geq 4$ have one
of these two graphs as a minor, these graphs also have $K_{2,4}$ minors.
 Hence at most one of $k,\ell,m$ can be greater than 4.
 When $k= 6$ and $\ell=m=4$, there is again a $K_{2,4}$ minor, shown on the right in
Figure~\ref{fig:S_1m}. Graphs with $k >6$ and $\ell=m=4$ have this
graph as a minor and hence also have a $K_{2,4}$ minor.
 For $k=5$ and $\ell=m=4$, we have the graph shown on the left and middle
in Figure~\ref{fig:S_3m}. With no dashed edges of the common triangle,
this graph is isomorphic to $B$.
 With at least one dashed edge there is a $K_{2,4}$ minor as shown on
the left of the figure for $e_1$ ($e_2$ is symmetric) or in the middle
for $e_3$.
 Hence $k=\ell=m=4$, and we have the graph shown on the right in
Figure~\ref{fig:S_3m}.  With any two dashed edges, the graph has a
$K_{2,4}$ minor, shown in the figure for $e_1$ and $e_2$. With no or one
dashed edge, the graph is isomorphic to $K_{3,3}$ or $A$, respectively.

\begin{figure}[h]
	\centering \scalebox{.9}{\includegraphics{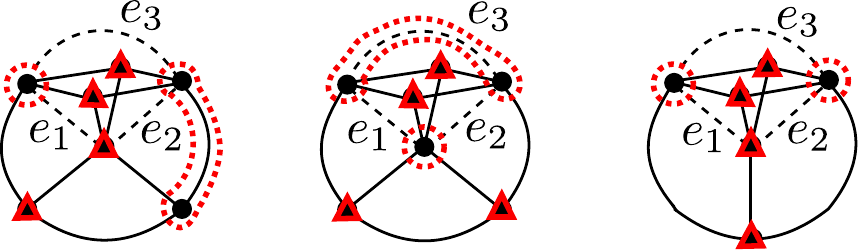}}
	\captionof{figure}{}
	\label{fig:S_3m}
\end{figure}
       
Next consider a common $3$-sum of two wheels and a prism.  If the wheels
are $W_5$ and $W_4$, then all common $3$-sums have the $K_{2,4}$ minor
shown on the left in Figure~\ref{fig:S_5m}.  Any other combination of
wheels gives this, and hence $K_{2,4}$, as a minor, unless both wheels
are $W_4$.  Then we have the graph shown on the right in
Figure~\ref{fig:S_5m}. With any dashed edge we have a $K_{2,4}$ minor,
shown in the figure for $e_1$. With no dashed edges, the graph is
isomorphic to $C$.

\begin{figure}[h]
	\centering \scalebox{.8}{\includegraphics{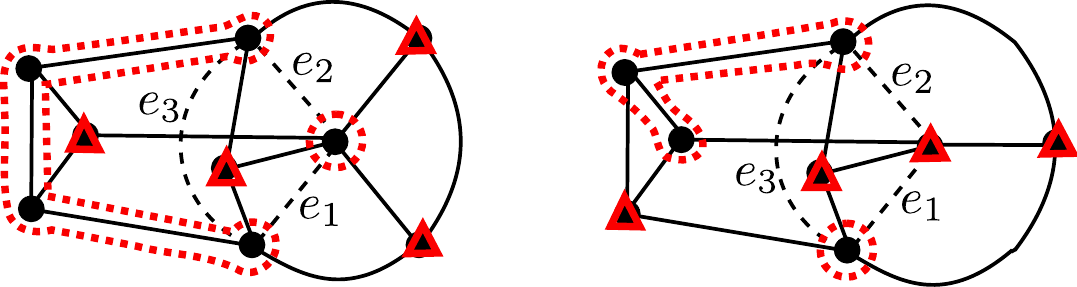}}
	\captionof{figure}{}
	\label{fig:S_5m}
\end{figure}

 Now consider a common $3$-sum of two wheels $W_k$ and $W_{\ell}$. 
Suppose the hubs of the wheels are not identified, or $k=4$ or $\ell=4$.
 We have the graph shown on the left in Figure~\ref{fig:S_7m}. At least
one of the edges labeled $e_1$ and $e_2$ must be present in the common
$3$-sum to ensure there are no degree $2$ vertices. Let $n=k+\ell-3$.
With $e_1$ and $e_2$, the graph is isomorphic to $G_{n,k-2,\ell-2}$.
With $e_1$ (or symmetrically $e_2$), the graph is isomorphic to either
$G_{n,k-3,\ell-2}$ or $G_{n,k-2,\ell-3}$. In all cases $e_3$ is the
optional plus edge. The spine is shown in the figure as the thick,
highlighted path. Hence we obtain graphs in $\Gi$.

 Now suppose that $k, \ell \ge 5$ and the hubs of $W_k$ and $W_{\ell}$
are identified in the common $3$-sum.
 The graph with $k = \ell = 5$ appears on the right in 
Figure~\ref{fig:S_7m}.
 With the edge labeled $e_1$, we have the $K_{2,4}$ minor shown,
and if $k, \ell \ge 5$ we get a similar minor.
 Without $e_1$, both $e_2$ and $e_3$ must be present to ensure there are
no vertices of degree $2$, and the graph is isomorphic in the general
case to $W_{k+\ell-3} \in \Gi$.

\begin{figure}[h]
	\centering \scalebox{.8}{\includegraphics{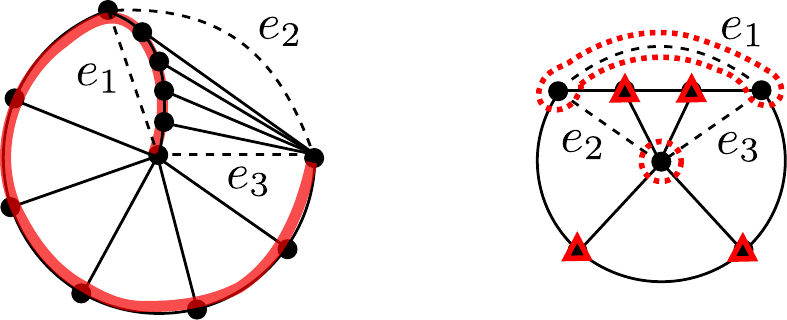}}
	\captionof{figure}{}
	\label{fig:S_7m}
\end{figure}
	
 Now consider a common $3$-sum of two prisms and one wheel. For $W_4$ we
have the graph on the left in Figure~\ref{fig:S_9m} with the $K_{2,4}$
minor shown; for any larger wheel we get this graph, and hence
$K_{2,4}$, as a minor.

 Next consider a common $3$-sum of two or three prisms.  For two prisms
we have the graph on the right in Figure~\ref{fig:S_9m}. At least two
dashed edges are needed to prevent a degree $2$ vertex and so we have
the $K_{2,4}$ minor shown. In a common $3$-sum of three prisms, the
dashed edges need not be present to ensure $3$-connectivity. However,
instead of using one of the dashed edges in the $K_{2,4}$ minor as on
the right in Figure~\ref{fig:S_9m}, we can use a path between these two
vertices through the third prism. Hence a similar $K_{2,4}$ minor
exists.

\begin{figure}[h]
	\centering \scalebox{.8}{\includegraphics{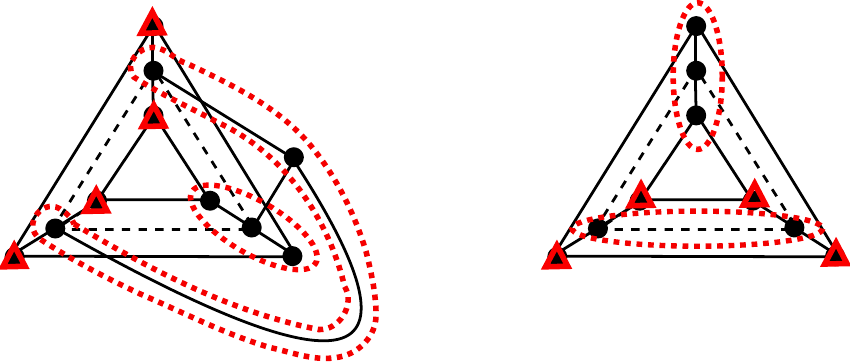}}
	\captionof{figure}{}
	\label{fig:S_9m}
\end{figure}

 Consider a common $3$-sum of one wheel $W_k$ and one prism; this is
unique up to isomorphism.  Figure \ref{fig:S_11m} shows the graph for
$k=5$ on the left.
 To prevent vertices of degree $2$, either $e_1$ is present, in which
case we have the $K_{2,4}$ minor shown, or the other two dashed edges
must exist, and the graph is isomorphic to $G_{8,2,4}$.
 For $k \ge 6$ there is a similar minor or the graph is isomorphic to
$G_{k+3, 2, k-1}$.
 The graph for $k=4$ is shown on the right in Figure~\ref{fig:S_11m}. At
least two dashed edges must be present to prevent degree $2$ vertices.
With two or three dashed edges the graph is isomorphic to $G_{7,2,3}$ or
$D$, respectively.

\begin{figure}[h]
	\centering \scalebox{.9}{\includegraphics{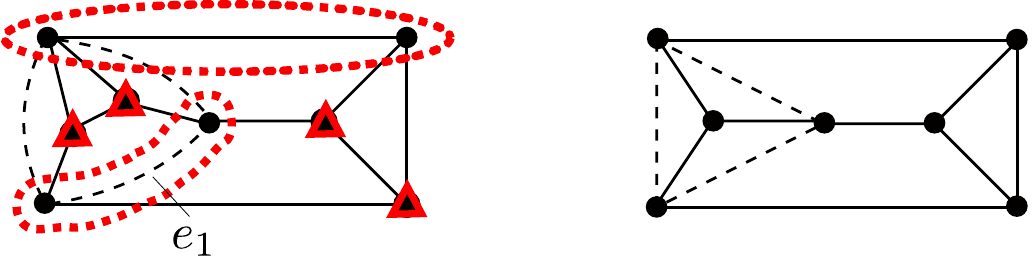}}
	\captionof{figure}{}
	\label{fig:S_11m}
\end{figure}

Finally, a common $3$-sum of a single graph is $W_k \isom G_{k,1,k-3}^+
\in \G$ or the triangular prism, isomorphic to $G_{6,2,2} \in \G$.
 \end{proof} 

In \cite{ding} Ding and Liu also prove the following result, where
$K_{3,3}^{\ddag}$ is the graph $K_{3,3}$ with two additional edges added
on the same side of the bipartition.

 \begin{theorem}[Ding and Liu \cite{ding}]  The family of all
$3$-connected $K_{3,3}^{\ddag}$-minor-free graphs consists of
$3$-connected planar graphs and $3$-connected minors of three
small graphs on at most ten vertices.
 \label{thm:K_3,3dag}
 \end{theorem}

 Because $K_{2,4}$ is a subgraph of $K_{3,3}^{\ddag}$,
$K_{2,4}$-minor-free graphs must be a subset of the graphs described in
Theorem~\ref{thm:K_3,3dag}.  Combining this with
Theorem~\ref{thm:Oct-e}, we conclude that all large enough
$K_{2,4}$-minor-free graphs $G$ must be planar (and so
$K_{3,3}$-minor-free) members of $\mathcal{S}$, hence common $3$-sums of
at most two graphs, which reduces the work needed to conclude that $G
\in \Gi$.
 The analysis required for small graphs is not simplified by using
Theorem~\ref{thm:K_3,3dag}, however, so we provide the full analysis
using only Theorem~\ref{thm:Oct-e}.


\section{The $2$-connected case}\label{2connected}

We begin this section by looking at how $K_{2,t}$ minors interact with
separations in a graph. We will mostly be concerned with
$2$-separations. 

\begin{lemma}
 Suppose $(H,J)$ is a $2$-separation in a graph $G$ with $V(H) \cap V(J)
= \{x,y\}$. If $G$ contains a standard $K_{2,t}$ minor $(\ptwo_1,
\ptwo_2; S)$ with $t \geq 3$, then one of the following hold:

 \hangitem (i) there exists a $K_{2,t}$ minor in $H+xy$,
 \hangitem (ii) there exists a $K_{2,t}$ minor in $J+xy$, or
 \hangitem (iii) $x \in \ptwo_1$ and $y \in \ptwo_2$ (or vice versa).
 \label{lem:lemma_s}
\end{lemma}

\begin{proof} Let $H'=H\vdel\{x,y\}$ and $J'=J\vdel\{x,y\}$.
 Assume (iii) does not hold, then $\{x,y\} \cap R_i = \emptyset$ for at
least one $i$; we may suppose that $\{x,y\} \cap R_2 = \emptyset$. 
 Since $R_2$ induces a connected subgraph, this means that $R_2 \subseteq
V(H')$ or $R_2 \subseteq V(J')$; without loss of generality we assume
that $R_2 \subseteq V(H')$.  Then necessarily $S \subseteq V(H)$, and so
$(R_1 \cap V(H), R_2; S)$ is a standard $K_{2,t}$ minor in $H+xy$ and
(i) holds.
 \end{proof}

 By a $K_{2,t}$ minor $(R_1,R_2;S)$ \textit{rooted at $x$ and $y$}, we
mean $x \in R_1$ and $y \in R_2$. If part (iii) of
Lemma~\ref{lem:lemma_s} holds, then the $K_{2,t}$ minor splits into two
minors, $K_{2,t_1}$ and $K_{2,t_2}$ with $t_1+t_2=t$, both rooted at $x$
and $y$.  For $K_{2,4}$ minors this means that we will be concerned with
rooted $K_{2,2}$ minors; we will describe the structure of graphs
without rooted $K_{2,2}$ minors.  Note that Demasi \cite[Lemma
2.2.2]{demasi} has characterized graphs without $K_{2,2}$ minors rooted
at all four vertices, in terms of disjoint paths.

 An \textit{$xy$-outerplane embedding} of a connected graph $G$ with $x,
y \in V(G)$ is an embedding of $G$ in a closed disk $D$ such that a
hamilton $xy$-path $P$ of $G$ is contained in the boundary of $D$. 
 This is equivalent to embedding $G$ in the plane so that the outer
facial walk contains $P$ as an uninterrupted subwalk, or so that all
edges not in $P$ lie `on the same side' of $P$; we use this as our
practical definition.
 The path $P$ is called the \textit{outer path}. A graph is
\textit{$xy$-outerplanar}, or generically \textit{path-outerplanar}, if
it has an $xy$-outerplane embedding.

 A \textit{block} is a connected graph without a cutvertex: an isolated
vertex, an edge, or a $2$-connected graph.
 The \textit{blocks of a graph $G$} are the maximal blocks that are
subgraphs of $G$.
 The \textit{block-cutvertex tree} of a connected graph $G$ is a tree
whose vertices are the blocks and cutvertices of $G$; a block $B$ and
cutvertex $v$ are adjacent if $v \in V(B)$.

 The following useful properties are obvious, so we omit their proofs.

 \begin{lemma}\label{lem:xy-op-prop}
 (i) If $G$ is $xy$-outerplanar, $H$ is $yz$-outerplanar, and $V(G)
\cap V(H) = \{y\}$ then $G \cup H$ is $xz$-outerplanar.

 \smallskip\noindent
 (ii) Suppose $x \ne y$.  Then $G$ is $xy$-outerplanar if and only if
$G+xy$ is a block with an outerplane embedding in which $xy$ is on the
outer face.  Such an embedding of $G+xy$ is also $xy$-outerplane.
 \end{lemma}

 We now characterize rooted $K_{2,2}$-minor-free graphs.

 \begin{lemma} Suppose $x$ and $y$ are distinct vertices of $G$ and
$G'=G+xy$ is a block.  Then $G$ has no $K_{2,2}$ minor rooted at $x$ and
$y$ if and only if $G$ is xy-outerplanar.
 \label{lem:rooted22}
 \end{lemma}

 \setcounter{acase}{0}
 \begin{proof}
 $(\implied)$ Assume an $xy$-outerplane embedding of $G$.  Add a
vertex $z$ and edges $xz$, $yz$ to $G$ in the outer face; the resulting
graph $G''$ is outerplanar.  If $G$ has a $K_{2,2}$ minor rooted at $x$
and $y$, then $G''$ has a $K_{2,3}$ minor, which is a contradiction
since outerplanar graphs are $K_{2,3}$-minor-free.

 \smallskip\noindent
 $(\implies)$ Proceed by induction on $|E(G)|$.  The base case
for $G$ is $K_2$ which has no $K_{2,2}$ minor rooted at $x$ and $y$ and
is clearly $xy$-outerplanar.  Now assume the claim holds for all graphs
on at most $m \geq 1$ edges and suppose $|E(G)|=m+1$.  Then $G'$ is
$2$-connected.

 First assume there is a cutvertex $v$ in $G$.
 Since $G'$ is $2$-connected, the block-cutvertex tree of $G$ must be a
path $B_1v_1B_2v_2\ldots v_{k-1}B_k$ where $k \ge 2$, $x \in V(B_1)-\{v_1\}$
and $y \in V(B_k)-\{v_{k-1}\}$.  Define $v_0=x$ and $v_k=y$.  Because
$G$ has no $K_{2,2}$ minor rooted at $x$ and $y$, each block $B_i$ has
no $K_{2,2}$ minor rooted at $v_{i-1}$ and $v_i$ for $1 \leq i \leq k$. 
Thus, by induction each block $B_i$ is $v_{i-1}v_i$-outerplanar.  By
Lemma \ref{lem:xy-op-prop}(i), the outerplane embeddings of the blocks
can then be combined to create an $xy$-outerplane embedding of $G$, as
in Figure~\ref{fig:2conn_9}.  

\begin{figure}[h]
\centering \scalebox{.8}{\includegraphics{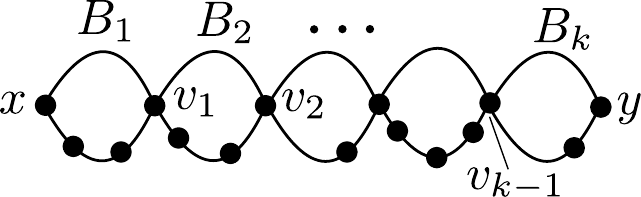}}
\caption{\label{fig:2conn_9}}
\end{figure}

 Now suppose $G$ has no cutvertex ($G$ is $2$-connected). 
 Assume that $G$ contains the edge $xy$. Then by induction, $G\edel xy$
has an $xy$-outerplane embedding. By Lemma \ref{lem:xy-op-prop}(ii),
$G$ also has an $xy$-outerplane embedding. Therefore, we may assume that
$G$ does not contain the edge $xy$. Since $G$ has no cutvertex, there
exist two internally disjoint $xy$-paths. Since $xy \notin E(G)$, each
path has an internal vertex, and hence they yield a $K_{2,2}$ minor
rooted at $x$ and $y$, a contradiction.
 \end{proof}

In order to describe the structure of $2$-connected $K_{2,4}$-minor-free
graphs, we need the following lemma:

 \begin{lemma} Suppose $t \ge 3$.  Let $z$ be a degree $2$ vertex in a
graph $G$ with neighbors $x$ and $y$. Let $G'$ be the graph formed from
$G$ by replacing the path $xzy$ with an $xy$-outerplanar graph $J$ on at
least three vertices. Then $G$ is $K_{2,t}$-minor-free if and only if
$G'$ is $K_{2,t}$-minor-free.  
 \label{lem:lemma_r}
 \end{lemma}

\begin{proof} $(\implied)$  $G$ is a minor of $G'$ so if $G'$ is
$K_{2,t}$-minor-free then so is $G$.

 \smallskip\noindent
 $(\implies)$ Let $H=G\vdel z$.  Then $(H,J)$ is a $2$-separation in
$G'$ with $V(H) \cap V(J)=\{x,y\}$. Because $G$ is $K_{2,t}$-minor-free,
we know that $H+xy$ is $K_{2,t}$-minor-free and also there is no
$K_{2,t-1}$ minor in $H$ rooted at $x$ and $y$. Because $J+xy$ is
outerplanar, $J+xy$ is $K_{2,t}$-minor-free. Thus by
Lemma~\ref{lem:lemma_s}, if $G'$ has a $K_{2,t}$ minor, then $x \in R_1$
and $y \in \ptwo_2$. If $|S \cap V(J)| \geq 2$, then $J$ has a $K_{2,2}$
minor rooted at $x$ and $y$ which contradicts Lemma~\ref{lem:rooted22}.
Thus $|S \cap V(H)| \geq t-1$ but now we have a $K_{2,t-1}$ minor rooted at
$x$ and $y$ in $H$ which is a contradiction. Hence $G'$ is
$K_{2,t}$-minor-free.  \end{proof}      
  
 We can now describe the structure of $2$-connected $K_{2,4}$-minor-free
graphs using one new concept.
 If $G$ is $K_{2,4}$-minor-free then $F \subseteq E(G)$ is
\textit{subdividable} if the graph formed from $G$ by subdividing all
edges of $F$ (replacing each edge by a path of length $2$) is
$K_{2,4}$-minor-free.
 The edge $e$ is \textit{subdividable} if $\{e\}$ is subdividable.
 If $F$ is a subdividable set then every edge of $F$ is subdividable,
but the converse is not true.

 \begin{theorem}
 Let $G$ be a block. Then $G$ is
$K_{2,4}$-minor-free if and only if one of the following holds.

 \hangitem (i) $G$ is outerplanar.

 \hangitem (ii) $G$ is the union of three $xy$-outerplanar graphs $H_1,
H_2, H_3$ and possibly the edge $xy$, where $|V(H_i)| \ge 3$ for each
$i$ and $V(H_i) \cap V(H_j) = \{x,y\}$ for $i \ne j$.

 \hangitem (iii) $G$ is obtained from a $3$-connected
$K_{2,4}$-minor-free graph $G_0$ by replacing each edge $x_iy_i$ in a
(possibly empty) subdividable set of edges $\{x_1y_1,x_2y_2,\allowbreak
\ldots,x_ky_k\}$ by an $x_i y_i$-outerplanar graph $H_i$, where $V(H_i)
\cap V(G_0) = \{x_i, y_i\}$ for each $i$, and $V(H_i) \cap V(H_j)
\subseteq V(G_0)$ for $i \ne j$.
 \label{lem:2connclass}
 \end{theorem}

\begin{proof} $(\implied)$  For (i), all outerplanar graphs are
$K_{2,4}$-minor-free since they are $K_{2,3}$-minor-free.  To show that
a graph $G$ in (ii) is $K_{2,4}$-minor-free, we use
Lemma~\ref{lem:lemma_r}.  $G$ is $K_{2,4}$-minor-free if the graph
formed from $G$ by replacing each of the three outerplanar pieces with a
single vertex is $K_{2,4}$-minor-free.  This graph is either $K_{2,3}$
or $K_{1,1,3}$ and is thus $K_{2,4}$-minor-free.  We use
Lemma~\ref{lem:lemma_r} again to show that graphs in (iii) are
$K_{2,4}$-minor-free.  Let $G'$ be formed from a $3$-connected
$K_{2,4}$-minor-free graph by subdividing a subdividable set of edges. 
$G'$ is still $K_{2,4}$-minor-free by the definition of subdividable
set.  Now replace each subdivided edge $x_i z_i y_i$ with an
$x_iy_i$-outerplanar graph; by Lemma~\ref{lem:lemma_r}, the resulting
graph is still $K_{2,4}$-minor-free.

 \smallskip\noindent
 $(\implies)$ Suppose $G$ is a $K_{2,4}$-minor-free block.  We proceed
by induction on $n=|V(G)|$.  As the basis, if $n \le 4$ then $G$ is one
of $K_1$, $K_2$, $K_3$, $K_{1,1,2}$ or $C_4$, which are outerplanar and
covered by (i), or $K_4$, which is $3$-connected and covered by (iii). 
If $G$ is $3$-connected then (iii) holds.

 So we may assume that $n \ge 5$ and $G$ has a $2$-cut $\{x,y\}$.  Let
$H'_1, H'_2, \ldots, H'_\ell$, where $\ell \ge 2$, be the components of
$G\vdel \{x,y\}$, and for each $i$ let $H_i$ be the subgraph induced by
$V(H'_i) \cup \{x,y\}$.

 If $\ell \ge 4$, then $G$ has a $K_{2,4}$ minor with $x \in \ptwo_1$, $y
\in \ptwo_2$, and $S$ consisting of a vertex from each of $H'_1, H'_2,
H'_3, H'_4$.  This is a contradiction.

 Suppose $\ell=3$. 
 If some $H_i$ is not $xy$-outerplanar, then we have a $K_{2,4}$ minor:
by Lemma~\ref{lem:rooted22}, there is a $K_{2,2}$ minor rooted at $x$
and $y$ in $H_i$, to which we may add one vertex from each of the two
other components of $G\vdel\{x,y\}$.  Thus, $H_1, H_2, H_3$ are all
$xy$-outerplanar and (ii) holds.

 \def\Gprime{H^+_1}
 Now suppose $\ell=2$. If neither $H_1$ nor $H_2$ is $xy$-outerplanar, then $G$
contains a $K_{2,4}$ minor. If both are $xy$-outerplanar then $G$ is
outerplanar as in (i). Hence one, say $H_1$, is not $xy$-outerplanar and
the other, $H_2$, is $xy$-outerplanar.
 Let $\Gprime = H_1+xy$.
 Since $|V(\Gprime)| <
|V(G)|$, by induction $\Gprime$ is in (i), (ii), or (iii).
 By Lemma \ref{lem:xy-op-prop}(ii), because $H_1$ is not
$xy$-outerplanar,  $\Gprime$ is not outerplanar and hence not in (i).
 If $\Gprime$ is in (ii), then the $2$-cut $\{u,v\}$ in $\Gprime$ giving
three components is also a $2$-cut in $G$ giving three components, and
so, applying the argument for $\ell=3$ to $\{u,v\}$, (ii) holds for $G$.

 Now assume $\Gprime$ is in (iii): $\Gprime$ is a $3$-connected
$K_{2,4}$-minor-free graph $G_0$ with each edge $f=uv$ of a
subdividable set $F$ replaced by a
$uv$-outerplanar graph $J(f)$.
 Let $H^*_2$ be $H_2+xy$ if $xy \in E(G)$, and $H_2$ otherwise.  In
either case $H_2^*$ is $xy$-outerplanar and $G$ is obtained from
$\Gprime$ by replacing $xy$ by $H_2^*$.

 Suppose first that $xy \notin \bigcup_{f \in F} E(J(f))$;
 then $xy \in E(G_0) \sdel F$.
 If we let $J(xy)=H_2^*$, then $G$ is obtained from $G_0$ by replacing
each $f \in F \cup \{xy\}$ by $J(f)$.
 The graph obtained from $G$ by replacing every $J(f)$, $f \in F \cup
\{xy\}$, by a path of length two with the same ends as $f$ is the same as
the graph obtained from $G_0$ by subdividing every edge of $F \cup
\{xy\}$.
 Since $G$ is $K_{2,4}$-minor-free, this graph is also
$K_{2,4}$-minor-free by repeated application of Lemma \ref{lem:lemma_r},
 so $F \cup \{xy\}$ is subdividable in $G_0$.
 Hence (iii) holds for $G$.

 Next suppose $xy$ is an edge of some $J(f)$, $f = uv \in F$, with outer
path $P$.
 Suppose $xy \notin E(P)$.
 Then there is at least one vertex in the subpath $Q$ of $P$ between,
but not including, $x$ and $y$.
 No vertex of $Q$ is adjacent to a vertex of $V(\Gprime)-V(J(f))$ or,
because $xy \in E(J(f))$, to a vertex of $V(J(f)) \sdel (V(Q) \cup
\{x,y\})$.
  Now there exists $w \in V(\Gprime)\sdel V(J(f))$, and $Q$ and $w$ are in
different components of $\Gprime\vdel\{x,y\}=H'_1$, contradicting the
fact that $H'_1$ is connected.
 So $xy \in E(P)$.  Then
the graph $J'(f)$ obtained by replacing $xy$ in
$J(f)$ with the $xy$-outerplanar graph $H_2^*$
is still $uv$-outerplanar. Thus $G$ is again in (iii).
 \end{proof}

 To complete the $2$-connected case, it remains to find all subdividable
sets of edges in part (iii) of Theorem~\ref{lem:2connclass} for each
$3$-connected $K_{2,4}$-minor-free graph.  If a set of edges is
subdividable, then all subsets of that set are also subdividable, so it
suffices to state the maximal (under inclusion) subdividable sets of
edges in each graph. We start with graphs in $\G$ with $n \geq 6$.
 The graphs $G_{6,2,2}$, $G_{6,2,2}^+\isom G_{6,2,3}$, and $G_{7,2,3}$
need special treatment and are dealt with later.

 \begin{theorem} Consider $G_{n,r,s}^{(+)} \in \G$ with $r \leq s$ and
$n \geq 6$.  (Results for $r > s$ may be obtained using the isomorphism
between $G_{n,r,s}^{(+)}$ and $G_{n,s,r}^{(+)}$.)

 \smallskip\noindent
 (i) When $r = 1$, the wheel $G_{n,1,n-3}^+$ has $n-1$ maximal
subdividable sets of edges. Each one includes all edges of the rim as
well as one of the spokes.

 \smallskip\noindent
 (ii) When $r=2$, $G_{n,2,s}$ with $s \ge 4$ or $G_{n,2,s}^{+}$ with $s
\ge 3$ has two maximal subdividable sets of edges: the edge sets
of
 the spine, $v_1v_2\ldots v_n$,
 and second spine, $v_{n-2}v_{n-3}\ldots v_1v_{n-1}v_n$. 

 \smallskip\noindent
 (iii) When $r \ge 3$ the only maximal subdividable set of edges is the
edge set of the spine, $v_1v_2\ldots v_n$.
 \label{lem:subG}
 \end{theorem}

 \begin{proof}
 We first show that each claimed subdividable set is subdividable.  For the
wheel $G_{n,1,n-3}^+$, subdividing all edges of the rim and one spoke
gives a graph isomorphic to a subgraph of $G_{2n,2,2n-4}$, and hence
$K_{2,4}$-minor-free.  For $r \ge 2$ the graph formed by subdividing all
edges of the spine in $G_{n,r,s}^{(+)}$ is isomorphic to a subgraph of
another graph in $\G$ with $2n-1$ vertices, and thus
$K_{2,4}$-minor-free.  So the edge set of the spine is subdividable.
 When $r=2$ the second spine is the image under an isomorphism of the
spine in another (or possibly the same) member of $\G$, and hence the
edge set of the second spine is also subdividable.

 \smallskip
 Now we show that the sets of edges listed are maximal and are the only
subdividable sets. Begin with the wheel $G_{n,1,n-3}^+$. All edges of
the rim are in each set so we consider the spokes. If we subdivide two
adjacent spokes, we have the $K_{2,4}$ minor shown on the left in
Figure~\ref{fig:wheel1m}. A similar minor exists if we subdivide
nonadjacent spokes as long as $n \geq 6$. Hence we cannot divide two
spokes and the sets listed are maximal and are the only subdividable
sets of edges.   

\begin{figure}[h]
	\centering \scalebox{.55}{\includegraphics{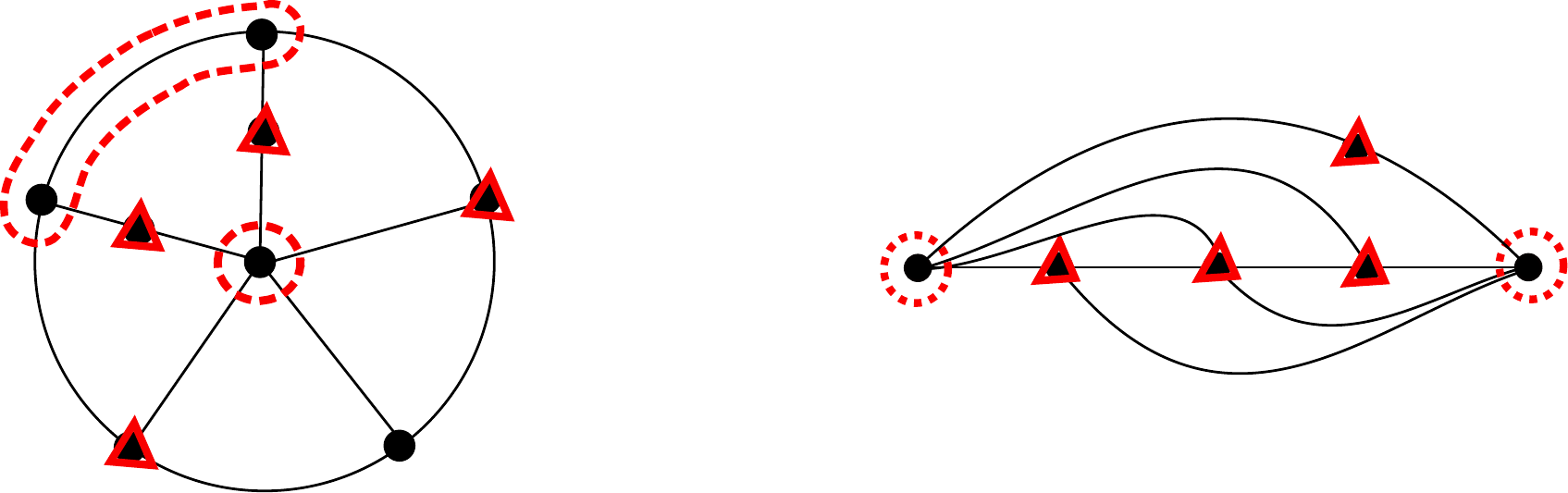}}
	\caption{\label{fig:wheel1m}} 
\end{figure}  
    
 Now assume $r,s \geq 2$.  For this portion of the proof, we remove the
assumption that $r \leq s$ (which is just for brevity in stating our
results).
 Denote by $G\circ e$ the graph formed from $G$ by subdividing the
edge $e$.
 We consider subdivision of non-spine edges $v_1v_{n-i}$ for $0 \leq i
\leq r$; edges $v_nv_{1+j}$ for $0 \leq j \leq s$ are handled by symmetry.
 The situations $i=0$ and $j=0$ correspond to a plus edge.

 We describe two cases in which we can find a $K_{2,4}$ minor. The
first, Case A, is the $K_{2,4}$ minor in $G_{5,2,2}^+\circ v_1v_5$ shown
on the right in Figure~\ref{fig:wheel1m}.
 If $s \ge 2$ and $0 \le i \le r-2$, then we form $G_{5,2,2}^+\circ
v_1v_5$, and hence $K_{2,4}$, as a minor from $G_{n,r,s}^{(+)}\circ
v_1v_{n-i}$ by contracting all edges of the paths $v_3v_4\ldots
v_{n-i-2}$ and $v_{n-i}v_{n-i+1}\ldots v_n$ and deleting multiple edges.

 The second case, Case B, is the $K_{2,4}$ minor in $G_{5,2,2}^+\circ
v_1v_3$ shown on the left in Figure~\ref{fig:g_2m}.
 Note that the minor does not use the edge $v_2v_5$. As with Case A,
this minor is inherited by the following larger graphs that have
$G_{5,2,2}^+\circ v_1v_3$ as a minor:

 \begingroup
 \hangitem (B1) $G_{n,r,s}^+\circ v_1v_{n-i}$ with $s \geq 2$ and $2
\leq i \leq r$;
 \hangitem (B2) $G_{n,r,s}\circ v_1v_{n-i}$ with $s \geq 2$ and $3 \leq
i \leq r$; and
 \hangitem (B3) $G_{n,r,s}^{(+)}\circ v_1v_{n-2}$ with $s \geq 3$.
 \endgroup

 \noindent
 For graphs in (B1), form $G_{5,2,2}^+\circ v_1v_3$ as a minor from
$G_{n,r,s}^+ \circ v_1v_{n-i}$ by contracting all edges of the paths
$v_3v_4\ldots v_{n-i}$ and $v_{n-i+1}v_{n-i+2}\ldots v_{n-1}$ and
deleting multiple edges as well as the edge $v_1v_3$ if it is present
after contraction.
 Similarly for graphs in (B2), contract all edges of the paths
$v_3v_4\ldots v_{n-i}$ and $v_{n-i+2}v_{n-i+3}\ldots v_n$ and delete
multiple edges and $v_1v_3$.
 For graphs in (B3), contract $v_1v_2$ and all edges of the path
$v_4v_5\ldots v_{n-2}$ and delete multiple edges and $v_1v_3$.

\begin{figure}[h]
	\centering \scalebox{.55}{\includegraphics{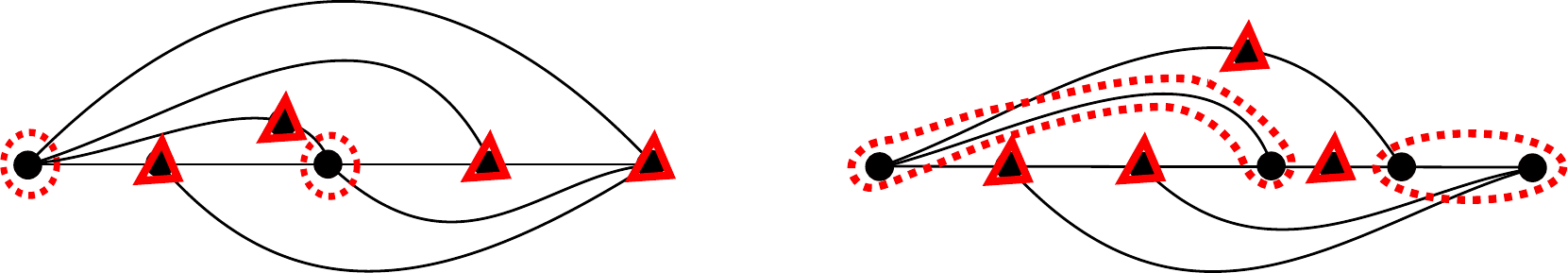}}
	\captionof{figure}{}
	\label{fig:g_2m} 
\end{figure}   

 For $G_{n,r,s}^{(+)}$ with $r,s \geq 3$, Case A shows that $v_1
v_{n-r-2}, v_1 v_{n-r-1}, \ldots, v_1 v_{n-1}$ and (if present) $v_1 v_n$
are not subdividable.
 Case B shows that $v_1v_{n-r},v_1v_{n-r+1},\ldots,v_1v_{n-2}$ are not
subdividable.
 By symmetry all non-spine edges incident to $v_n$ are not subdividable,
and hence the spine is the only maximal subdividable set of edges.

 Now either $r=2$ or $s=2$.  For our stated result we only need the case
$r = 2$ with $s$ as in (ii).
 Consider $G_{n,2,s}^{(+)}$. Case A forbids subdivision of $v_1v_n$ (if
present) and
(B3) does the same for $v_1v_{n-2}$.
 Applying symmetry, Case A forbids subdivision of $v_n v_{1+j}$ for $1
\le j \le s-2$, and Case B covers $v_n v_{1+j}$ for $2 \le j \le s$ if
there is a plus edge and $3 \le j \le s$ otherwise.  The conditions in
(ii) mean that these cover $v_n v_{1+j}$ for all $j$, $1 \le j \le s$.
 So the only possible subdividable non-spine edge is $v_1 v_{n-1}$,
which we already know is subdividable along with all edges of the spine
other than $v_{n-2} v_{n-1}$, as this is the edge set of the second
spine.
 So consider $v_1 v_{n-1}$ and $v_{n-2} v_{n-1}$ together.
 We use the $K_{2,4}$ minor in $G_{6,2,2} \circ v_1v_5 \circ v_4v_5$
shown on the right in Figure~\ref{fig:g_2m}.
 When $n \ge 6$, $G_{n,2,s}^{(+)} \circ v_1 v_{n-1} \circ
v_{n-2}v_{n-1}$ has $G_{6,2,2} \circ v_1v_5 \circ v_4v_5$, and hence
$K_{2,4}$, as a minor: delete $v_n v_{1+j}$ with $j=0$ (if present) and
$3 \le j \le s$, then contract all edges of $v_4 v_5 \ldots v_{n-2}$.
 Therefore $\{v_1 v_{n-1}, v_{n-2} v_{n-1}\}$ is not subdividable, and
the only maximal subdividable sets are the edge sets of the spine and
second spine.
 \end{proof}

\def\adjbox#1{\adjustbox{valign=c,margin=0pt 2pt}{#1}}
\begin{table}[h]
\caption{\label{table}}
\centering
\begin{tabular}{|c|c|c|}
\hline Graph   & Maximal Subdividable Sets of Edges &  Number of Symmetric Copies \\ \hline
\hline $K_4=W_4$ & \adjbox{\scalebox{.88}{\includegraphics{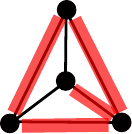}}} & 12 \\
\hline $\wheelfive$ ($\isom G_{5,2,2}$) & \adjbox{\scalebox{.8}{\includegraphics{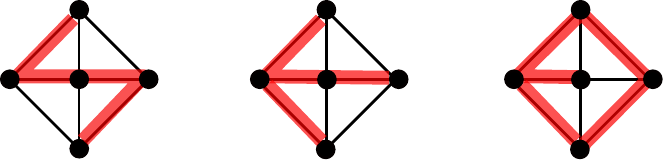}}} & 4 of each\\
\hline $G_{5,2,2}^+$ ($\isom K_5\edel e$) & \adjbox{\scalebox{.8}{\includegraphics{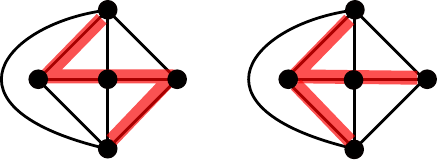}}} & 6 of each\\
\hline $G_{6,2,2}$ & edge set of spine & 6 \\
\hline $G_{6,2,2}^+$ ($\isom G_{6,2,3}$)
  & \adjbox{\pbox{3 in} {\quad edge set of spine in $G_{6,2,2}^+$ \\
                 edge set of second spine in $G_{6,2,2}^+$} }
  & \pbox{1 in} {\qquad spine: 1 \\ second spine: 2} \\
\hline $G_{7,2,3}$  & \adjbox{\pbox{3 in} {edge set of spine, edge set of second spine,\\
	\null\qquad\qquad $\{v_1v_2,v_4v_5,v_6v_7,v_3v_7\}$} } & 1 of each \\
\hline $K_5$  & $\emptyset$ & 1 \\
\hline $A, K_{3,3}$  & \adjbox{\scalebox{.8}{\includegraphics{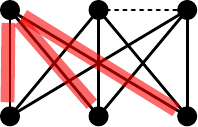}}} & \pbox{1in}{\quad $A$: 1 \\ $K_{3,3}$: 6} \\
\hline $A^+$  & \adjbox{\scalebox{.8}{\includegraphics{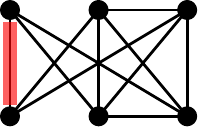}}} & 1 \\
\hline $B,B^+$  & \adjbox{\scalebox{.8}{\includegraphics{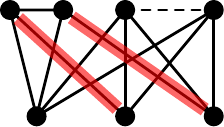}}} & 1 \\
\hline $C,C^+$  & \adjbox{\scalebox{.8}{\includegraphics{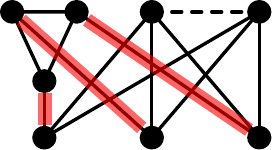}}} &  1 \\
\hline $D$  & \adjbox{\scalebox{.8}{\includegraphics{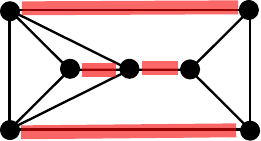}}} & 3 \\
\hline
\end{tabular}
\end{table}

 All remaining small graphs are covered by Table~\ref{table}.  Verifying
these results is straightforward; complete proofs may be found in 
\cite[Section 5.2]{dissertation}.  These results were also confirmed by
computer (the program may be obtained from the first author).
 The dashed edges in the table indicate edges present in one graph but
not the other.
 For example, in the row for $C$ and $C^+$, the dashed edge is present
in $C^+$ but not $C$.

 \begin{lemma} The maximal subdividable sets of edges for the nine small cases
not in $\G$ as well as $K_4 = W_4$, $W_5 \isom G_{5,2,2}$, $K_5\edel e
\isom G_{5,2,2}^+$, $G_{6,2,2}$, $G_{6,2,2}^+ \isom G_{6,2,3}$ and
$G_{7,2,3}$ are listed in Table~\ref{table}.
 \label{lem:subsmall}
 \end{lemma}

 As mentioned earlier, a graph $G$ is $K_{2,4}$-minor-free if and only
if each of its blocks is $K_{2,4}$-minor-free, so our overall result can
now be stated as follows.

 \begin{theorem}[Characterization of $K_{2,4}$-minor-free graphs]
 A graph is $K_{2,4}$-minor-free if and only if each of
its blocks is described by Theorem~\ref{lem:2connclass}, where for
Theorem~\ref{lem:2connclass} (iii), the $3$-connected graphs are given
in Theorem~\ref{thm:main24} and the subdividable sets are described in
Theorem~\ref{lem:subG} and Lemma~\ref{lem:subsmall}.  
 \label{thm:combin}
 \end{theorem}

 \section{Consequences}\label{consequences}

 Our characterization has a number of consequences.  First, as mentioned
in the introduction, we are interested in hamiltonian properties of
$K_{2,4}$-minor-free graphs.

 \begin{corollary} (i) Every $3$-connected $K_{2,4}$-minor-free graph
has a hamilton cycle.

 \noindent
 (ii) There are $2$-connected $K_{2,4}$-minor-free planar graphs that
have no spanning closed trail and hence no hamilton cycle.

 \noindent
 (iii) However, every $2$-connected $K_{2,4}$-minor-free graph has a
hamilton path.
 \end{corollary}

 \begin{proof} (i) The graph $G^{(+)}_{n,r,s} \in \G$ has a hamilton
cycle $(v_1 v_2 \ldots v_{s+1} v_n v_{n-1} \ldots v_{s+2})$.  The graphs
in Figure~\ref{fig:smallcasesnew} are also all hamiltonian.

 \smallskip
 \noindent
 (ii) A $2$-connected graph described by Theorem
\ref{lem:2connclass}(ii) is planar and has no closed spanning trail; the
simplest example is $K_{2,3}$.  (It is also possible to construct
examples using Theorem \ref{lem:2connclass}(iii).)

 \smallskip
 \noindent
 (iii) Define a \textit{hamilton base} in a graph to be a hamilton path
extended by a new edge at one or both ends, i.e., a trail of the form
$x_0 x_1 x_2 \ldots x_n x_{n+1}$, $x_1 x_2 \ldots x_n x_{n+1}$, or $x_1
x_2 \ldots x_n$, where $x_1 x_2 \ldots x_n$ is a hamilton path.
 If $B$ is a hamilton base in a graph $G_0$ and $G_1$ is obtained from
$G_0$ by subdividing the elements of a subset of $E(B)$ arbitrarily many
times, we observe that $G_1$ has a hamilton path.

 Now consider a $2$-connected $K_{2,4}$-minor-free graph $G$. If $G$ is
described by Theorem \ref{lem:2connclass}(i) then $G$ is hamiltonian. 
 Suppose $G$ is described by Theorem \ref{lem:2connclass}(iii), as
constructed from a $3$-connected $K_{2,4}$-minor-free graph $G_0$ by
replacing each edge of a subdividable set $S= \{x_1 y_1, x_2 y_2,
\ldots, x_k y_k\}$ by an $x_i y_i$-outerplanar graph.  Then $G$ has a
spanning subgraph $G_1$ which is obtained from $G_0$ by subdividing each
edge of $S$ some number of times.  If $G_0$ has a hamilton base
containing $S$, then $G_1$, and hence $G$, has a hamilton path.  So we
just need to verify that each maximal subdividable set of edges in $G_0$
is contained in a hamilton base.  Each subdividable set from Theorem
\ref{lem:subG} itself forms a hamilton base, and it is not difficult to
show that the subdividable sets from Lemma \ref{lem:subsmall} (Table
\ref{table}) are contained in hamilton bases; we omit the details.
 Finally, if $G$ is described by Theorem \ref{lem:2connclass}(ii) then
$G$ has a spanning subgraph $G_1$ obtained by subdividing edges of $G_0
= K_{2,3}$, and $K_{2,3}$ has a hamilton base containing all its edges,
so a similar argument applies.
 \end{proof}

 Second, a theorem of Dieng and Gavoille mentioned earlier can be
derived from our results.  We state it and just outline a proof.

 \begin{corollary}[Dieng and Gavoille, see {\cite[Th\'{e}or\`{e}me
3.2]{dieng}}]
 \label{cor:dg}
 For every $2$-connected $K_{2,4}$-minor-free graph $G$ there is $U
\subseteq V(G)$ with $|U| \le 2$ ($|U| \le 1$ if $G$ is planar)
such that $G\vdel U$ is outerplanar.
 \end{corollary}

 \begin{proof}[Sketch of proof]
 Consider the structure of $G$ as described in Theorem
\ref{lem:2connclass}.  If (i) holds no vertices need to be deleted, and
if (ii) holds then one of $x$ or $y$ can be deleted.  To verify the
result when (iii) holds, it suffices to show that for every
$3$-connected $K_{2,4}$-minor-free $G_0$ and every maximal subdividable
set of edges $F$ in $G_0$, there is $U \subseteq V(G_0)$ with $|U| \le
2$ ($|U| \le 1$ if $G_0$ is planar) so that $G_0 \vdel U$ has an
outerplane embedding with all remaining edges of $F$ (those not incident
with $U$) on the outer face. 
 If $G_0 = G_{n,r,s}^{(+)} \in \G$ is covered by Theorem
\ref{lem:subG} then $G_0-v_n$ always works.
 The result must be checked for the small graphs in Table \ref{table}.
 \end{proof}

 Dieng and Gavoille in fact showed that there is an $O(n)$ time
algorithm to find either a $K_{2,4}$ minor or a set $U$ as in Corollary
\ref{cor:dg} in any $n$-vertex graph.

 Third, our result also gives bounds on genus.

 \begin{corollary}
 \label{cor:genus}
 Every $2$-connected $K_{2,4}$-minor-free graph is either planar or else
toroidal and projective-planar.  Thus, its orientable and nonorientable
genus are at most $1$.
 \end{corollary}

 \begin{proof}
 The $3$-connected graphs described in Theorem \ref{thm:main24} are
planar or minors of $C^+$, and it is not difficult to find toroidal and
projective-planar embeddings of $C^+$.
 For connectivity $2$ the graphs $G$ constructed in Theorem
\ref{lem:2connclass} are either planar or have the same genus as
some $3$-connected $K_{2,4}$-minor-free graph $G_0$.
 \end{proof}

 Note that Corollary \ref{cor:genus} does not follow from Dieng and
Gavoille's result, Corollary \ref{cor:dg}, since a result of Mohar
\cite{Mo01apex} implies that graphs which become outerplanar after
deleting two vertices can have arbitrarily high (orientable) genus.

 Fourth, our result shows that the number of $3$-connected
$K_{2,4}$-minor-free graphs grows only linearly.
 For $n \ge 9$ the only such $n$-vertex graphs are those in $\Gi$, and
there are only $2n-8$ nonisomorphic such graphs.  Although we have not
done so, it should also be possible to deduce counting results for
$2$-connected $K_{2,4}$-minor-free graphs from our characterization.

 Finally, Chudnovsky, Reed and Seymour \cite{chud} showed that the
number of edges in a $3$-connected $K_{2,t}$-minor-free graph is at most
$5n/2 + c(t)$.  They provide examples to show that this is in a sense
best possible for $t \ge 5$.  Theorem \ref{thm:main24} shows that this
can be improved when $t=4$.  Using Theorem \ref{lem:2connclass} we can
also obtain a result for $2$-connected $K_{2,4}$-minor-free graphs.  We
omit the straightforward proofs, which use the fact that an $n$-vertex
outerplanar graph has at most $2n-3$ edges.

 \begin{corollary}
 %

 (i) Every $3$-connected $K_{2,4}$-minor-free $n$-vertex graph with $n
\ge 7$ has at most $2n-2$ edges, and such graphs with $2n-2$ edges exist
for all $n \ge 7$.  ($K_5$ has $2n$ edges, and $A^+$ has
$2n-1$ edges.)

 \noindent
 (ii) Every $2$-connected $K_{2,4}$-minor-free $n$-vertex graph with $n
\ge 6$ has at most $2n-1$ edges, and such graphs with $2n-1$ edges exist
for all $n \ge 6$.  ($K_5$ has $2n$ edges.)
 \end{corollary}

\end{document}